\documentclass[11pt]{amsart}
\usepackage{latexsym,amscd,amssymb, graphicx, shuffle, float}  
\usepackage[margin=1.0in]{geometry}

\newtheorem{theorem}{Theorem}
\newtheorem{proposition}[theorem]{Proposition}

\newtheorem{lemma}[theorem]{Lemma}

\newtheorem{problem}[theorem]{Problem}

\theoremstyle{definition}

\newcommand{\Cat}{{{\sf Cat}}}
\newcommand{\Park}{{{\sf Park}}}

\newcommand{\Ish}{{{\sf Ish}}}
\newcommand{\Rook}{{\sf{Rook}}}
\newcommand{\Shi}{{{\sf Shi}}}
\newcommand{\Cox}{{{\sf Cox}}}
\newcommand{\Rec}{{{\sf Rec}}}
\newcommand{\Stir}{{{\sf Stir}}}

\newcommand{\symm}{{\mathfrak{S}}}

\newcommand{\ZZ}{{\mathbb {Z}}}
\newcommand{\RR}{{\mathbb {R}}}

\newcommand{\A}{{\mathcal{A}}}

\begin{document}

\title[Bijections for the Shi and Ish  arrangements]
{Bijections for the Shi and Ish  arrangements}

\author{Emily Leven, Brendon Rhoades, and Andrew Timothy Wilson}
\address
{Deptartment of Mathematics \newline \indent
University of California, San Diego \newline \indent
La Jolla, CA, 92093-0112, USA}
\email{esergel@math.ucsd.edu, bprhoades@math.ucsd.edu, atwilson@math.ucsd.edu}

%\subjclass{}

\begin{abstract}
The {\sf Shi hyperplane arrangement} 
$\Shi(n)$
was introduced by Shi to study the Kazhdan-Lusztig
cellular structure of the affine symmetric group.  The {\sf Ish hyperplane arrangement} 
$\Ish(n)$
was introduced by Armstrong
in the study of diagonal harmonics.  Armstrong and Rhoades discovered 
a deep combinatorial similarity between the Shi and Ish arrangements.  We solve a collection
of problems posed by Armstrong \cite{A} and Armstrong and Rhoades \cite{AR} by giving  bijections
between regions of $\Shi(n)$ and $\Ish(n)$ which preserve certain statistics.
Our bijections generalize to the `deleted arrangements' $\Shi(G)$ and $\Ish(G)$ which depend
on a subgraph $G$ of the complete graph $K_n$ on $n$ vertices.
The key tools in our bijections are the introduction of an Ish analog of parking functions called 
{\sf rook words} and a new instance of the cycle lemma of enumerative combinatorics.  
\end{abstract}

\maketitle

\section{Introduction}
\label{Introduction}

In this paper we give the first bijective proofs of several equidistribution results due to
Armstrong \cite{A} and Armstrong and Rhoades \cite{AR}.  The sets underlying these results 
are regions of hyperplane arrangements related to ``Shi/Ish duality". 
We begin by defining these arrangements.

The {\sf Coxeter arrangement} $\Cox(n)$ of type A$_{n-1}$ (otherwise known as the braid arrangement)
is the linear arrangement of hyperplanes in $\RR^n$ given by
\begin{equation}
\label{coxeter-arrangement}
\Cox(n) := \{ x_i - x_j = 0 \,:\, 1 \leq i < j \leq n \},
\end{equation}
where $x_1, \dots, x_n$ are the standard coordinate functions on $\RR^n$.  
Recall that the {\sf regions} of a hyperplane arrangement $\A$ in $\RR^n$ are the connected 
components of the complement $\RR^n - \bigcup_{H \in \A} H$.
The regions of 
$\Cox(n)$ biject naturally with permutations in the symmetric group $\symm_n$ by letting 
a permutation $\pi = \pi_1 \dots \pi_n \in \symm_n$ correspond to the region of 
$\Cox(n)$ defined by the coordinate inequalities $x_{\pi_1} > \dots > x_{\pi_n}$.

The {\sf Shi arrangement} $\Shi(n)$ of type A$_{n-1}$ is the affine hyperplane arrangement
in $\RR^n$ given by
\begin{equation}
\label{shi-arrangement}
\Shi(n) := \Cox(n) \cup \{x_i - x_j = 1 \,:\, 1 \leq  i < j \leq n \}.
\end{equation}
The arrangement $\Shi(n)$ was introduced by J.-Y. Shi \cite{Shi}
to study the Kazhdan-Lusztig cells of the affine
symmetric group $\widetilde{\symm}_n$
and has  received a great deal of attention in algebraic and enumerative combinatorics.
Shi proved that the number of regions of $\Shi(n)$ is the famous expression $(n+1)^{n-1}$.
The Shi arrangement is a {\sf deformation} of the Coxeter arrangement in the sense of 
Postnikov and Stanley \cite{PS},
 i.e., the hyperplanes of $\Shi(n)$ are parallel to the hyperplanes of $\Cox(n)$.
 The arrangement $\Shi(3)$ is shown on the left in Figure~\ref{shiishthree}.

The {\sf Ish arrangement} $\Ish(n)$ is a different deformation of the Coxeter arrangement in $\RR^n$.
The hyperplanes of $\Ish(n)$ are given by
\begin{equation*}
\label{ish-arrangement}
\Ish(n) := \Cox(n) \cup \{ x_1 - x_j = i \,:\, 1 \leq  i < j \leq n \}.
\end{equation*}
As with the Shi arrangement, the initial motivation for the Ish arrangement was representation theoretic.
Armstrong defined $\Ish(n)$ to obtain an interpretation of the bounce statistic of Haglund as a statistic
on the type A root lattice and develop a connection between hyperplane arrangements
and diagonal harmonics.  The arrangement $\Ish(3)$ is shown on the right in Figure~\ref{shiishthree}.

Armstrong proved the following enumerative symmetry of the Shi and Ish arrangements.

\begin{figure}
\centering
\includegraphics[scale = 0.5]{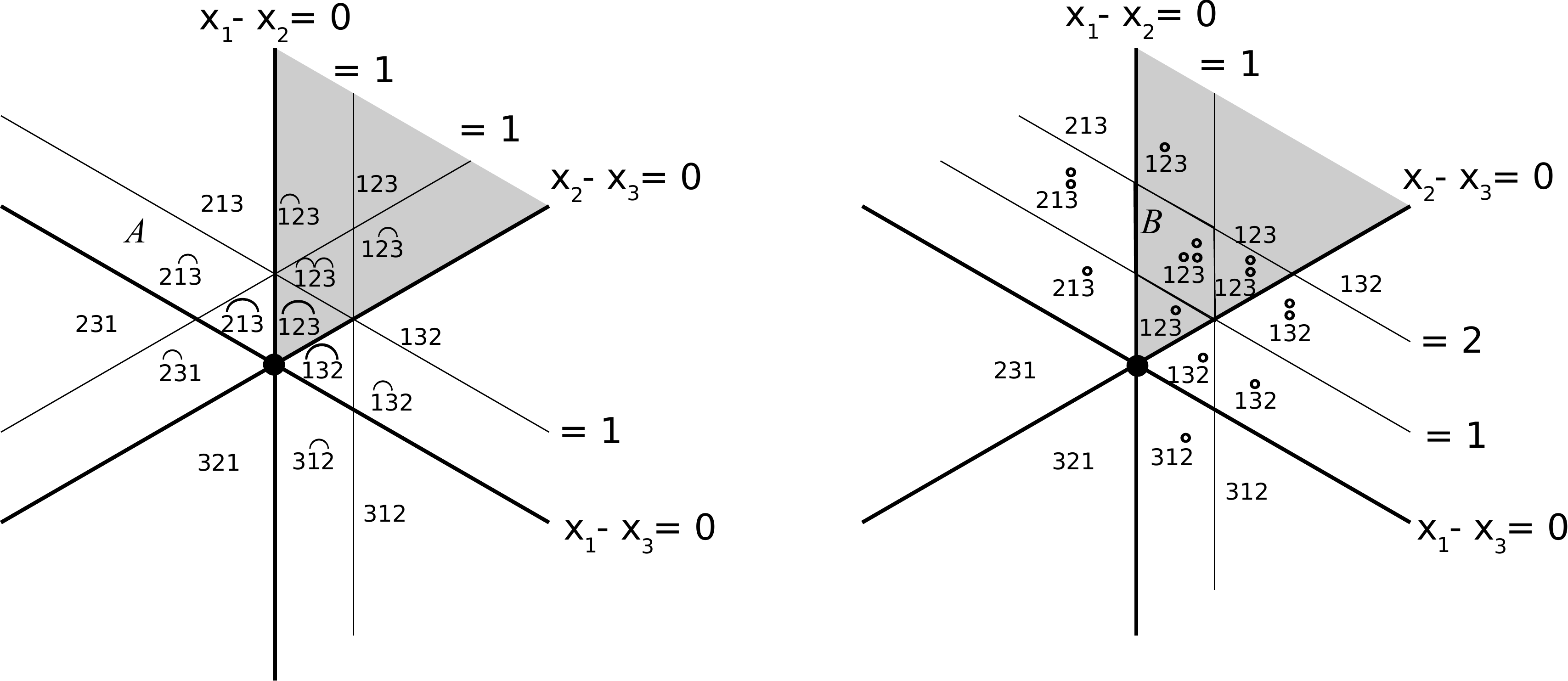}
\caption{The arrangements $\Shi(3)$ and $\Ish(3)$ labeled by Shi and Ish ceiling diagrams.}
\label{shiishthree}
\end{figure}

\begin{theorem} (Armstrong)
\label{shi-ish-basic-count}
The arrangements $\Shi(n)$ and $\Ish(n)$ have the same number of regions.
\end{theorem}

Armstrong proved Theorem~\ref{shi-ish-basic-count} by using the
finite fields method of Crapo and Rota \cite{CR} to show that the characteristic polynomial
$\chi_{\Ish(n)}(q)$ of the Ish arrangement is given by
$\chi_{\Ish(n)}(q) = q (q-n)^{n-1}$.  Headley \cite{H} proved that the characteristic polynomial 
$\chi_{\Shi(n)}(q)$ of the Shi arrangement is also equal to $q (q-n)^{n-1}$.  
It follows from Zaslavsky's Theorem \cite{Z} that 
$\Shi(n)$ and $\Ish(n)$ both have $(n+1)^{n-1}$ regions.
Armstrong posed the following problem.

\begin{problem} (Armstrong)
\label{basic-problem}
Find a bijection between the regions of $\Shi(n)$ and $\Ish(n)$.
\end{problem}

Theorem~\ref{shi-ish-basic-count} and Problem~\ref{basic-problem}
inspired Armstrong and Rhoades \cite{AR} to study the combinatorics of the 
Shi and Ish arrangements more deeply.
Their analysis starts with the basic observation  that the non-linear hyperplanes of $\Shi(n)$ and $\Ish(n)$
are in the following bijective correspondence:
\begin{equation}
\label{shi-ish-duality}
x_i - x_j = 1 \longleftrightarrow x_1 - x_j = i,
\end{equation}
for $1 \leq i < j \leq n$.  This correspondence is referred to as
``Shi/Ish duality" and preserves a remarkable amount of combinatorial information.
In order to state these equidistribution results, we will need some definitions.

Let $K_n$ denote the complete graph on the vertex set $[n] := \{1, 2, \dots, n\}$.
For any graph $G \subseteq {[n] \choose 2}$ on $[n]$ (which we identify with its edge set),
we define the {\sf deleted Shi and Ish arrangements} $\Shi(G)$ and $\Ish(G)$ to have 
hyperplanes
\begin{align}
\Shi(G) &= \Cox(n) \cup \{ x_i - x_j = 1 \,:\, (i < j) \in G \}, \\
\Ish(G) &= \Cox(n) \cup \{ x_1 - x_j = i \,:\, (i < j) \in G \}.
\end{align}
When $G = K_n$ we recover the standard Shi and Ish arrangements:
$\Shi(K_n) = \Shi(n)$ and $\Ish(K_n) = \Ish(n)$.  When $G = \emptyset$ is the graph
with no edges, we have 
$\Shi(\emptyset) = \Ish(\emptyset) = \Cox(n)$.  The arrangement $\Shi(G)$ was first 
considered by Athanasiadis and Linusson 
(and should not be confused with the `$G$-Shi arrangements' of Duval, Klivans, and Martin).
Figure~\ref{shiishrestrict} shows the arrangements $\Shi(G)$ and $\Ish(G)$ when $n = 3$
and $G$ is the ``path" $1 - 2 - 3$.
We describe two
statistics -- ``degrees of freedom" and ``ceiling partitions" -- defined on the regions
of the arrangements $\Shi(G)$ and $\Ish(G)$.

\begin{figure}
\centering
\includegraphics[scale = 0.5]{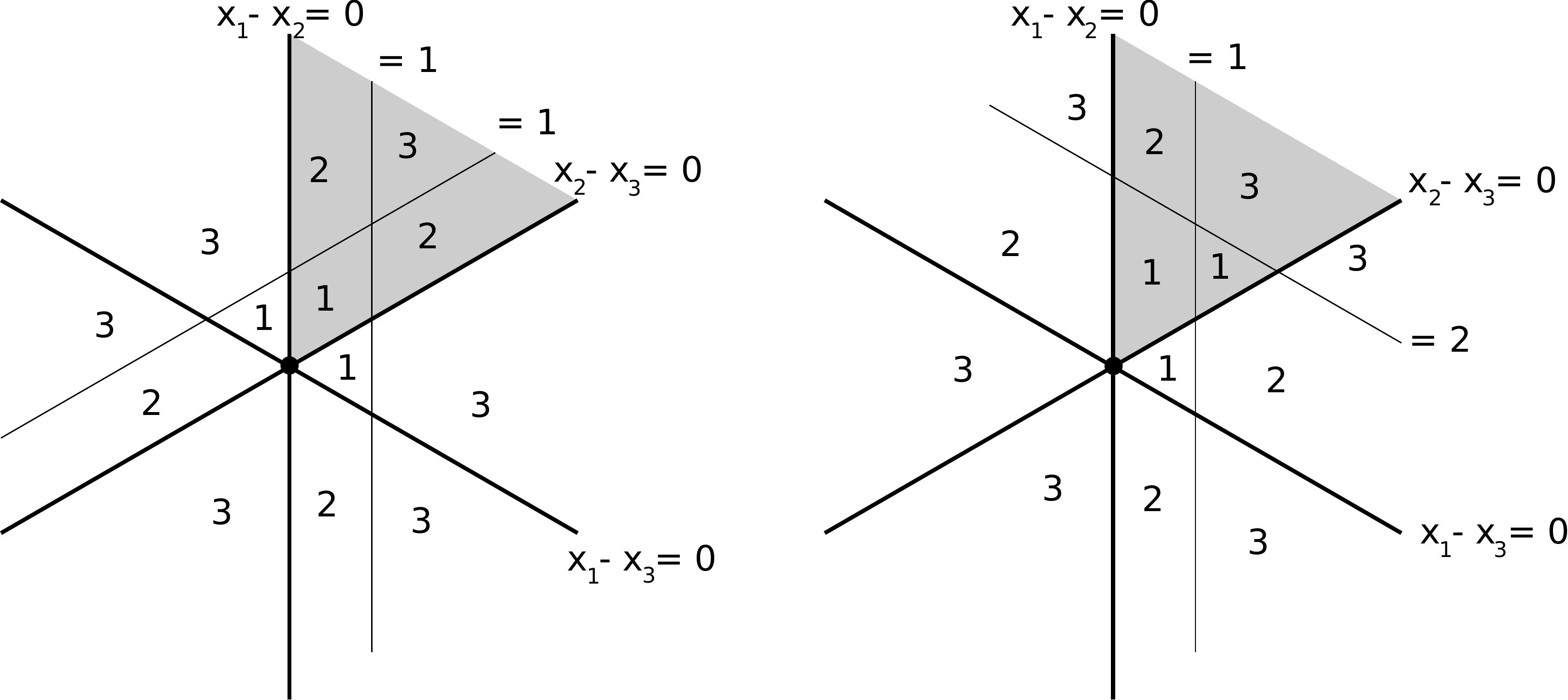}
\caption{The arrangements $\Shi(G)$ and $\Ish(G)$ for $G = 1 - 2 - 3$ labeled by degrees of freedom.}
\label{shiishrestrict}
\end{figure}

Given any hyperplane arrangement $\A$ in $\RR^n$ and any region $R$ of $\A$, the 
{\sf recession cone} $\Rec(R)$ of $R$ is
\begin{equation}
\Rec(R) := \{ v \in \RR^n \,:\, R + v \subseteq R \}.
\end{equation}
Since $R$ is convex, it follows that $\Rec(R)$ is a cone (i.e., closed under nonnegative
linear combinations).  If $\dim(\Rec(R)) = d$, we say that $R$ has 
{\sf $d$ degrees of freedom}.  Since $R$ is bounded if and only if $\Rec(R) = 0$, the degrees
of freedom statistic measures the failure of a region to be bounded.  

Figure~\ref{shiishrestrict} shows the regions of $\Shi(G)$ and $\Ish(G)$ labeled by the  
degrees of freedom statistic.  (Since the figure shows 2-dimension projections of 3-dimensional arrangements,
each region has an additional orthogonal degree of freedom.)

If $\A$ is any hyperplane arrangement in $\RR^n$ and $R$ is any region of $\A$, the hyperplanes
in $\A$ decompose the topological closure $\overline{R}$ into {\sf faces} of various dimensions.
A {\sf facet} of $R$ is a face of codimension $1$.  A non-linear hyperplane $H \in \A$ is a
{\sf ceiling} of $R$ if $H$ is the affine span of a facet of $R$ and $H$ does not separate $R$ from
the origin.

Let $\A$ be one of $\Shi(G)$ or $\Ish(G)$ and let $R$ be a region of $\A$.  The ceilings of $R$ 
can be used to define a set partition of $[n]$ called the {\sf ceiling partition} of $R$.  In particular,
if $\A = \Shi(G)$ (resp. $\Ish(G)$), the ceiling partition of $R$ is the set partition of $[n]$
generated by $i \sim j$ if $x_i - x_j = 1$ (resp. $x_1 - x_j = i$) is a ceiling of $R$.   
Since ceiling partitions are defined using affine hyperplanes, it follows
that 
every edge in the arc diagram of the ceiling partition of $R$ is contained in the graph $G$.

For example, if $R$ is the region of $\Shi(3)$ in Figure~\ref{shiishthree} marked $``A"$, then
the ceiling of $R$ is $x_1 - x_3 = 1$ and the ceiling partition is
$\{ \{1, 3\}, \{2\} \}$.  Also, if $R'$ is the region of $\Ish(3)$ in Figure~\ref{shiishthree}
marked $``B"$, then the ceilings of $R'$ are $x_1 - x_2 = 1$ and $x_1 - x_3 = 2$ and
the ceiling partition of $R'$ is $\{ \{1, 2, 3\} \}$.

Finally, recall that a region $R$ of a hyperplane arrangement containing $\Cox(n)$ is called
{\sf dominant} if the coordinate inequalities $x_1 > \dots > x_n$ are satisfied on $R$.
The dominant regions are shaded in Figures~\ref{shiishthree} and \ref{shiishrestrict}.
It is well known that the number of dominant regions of $\Shi(n)$ is the {\sf Catalan number}
$\Cat(n) = \frac{1}{n+1}{2n \choose n}$.  Armstrong proved that the number of dominant regions
of $\Ish(n)$ is also given by $\Cat(n)$.  More generally, Armstrong and Rhoades proved the 
following result \cite{AR}.

\begin{theorem} (Armstrong-Rhoades)
\label{dominance-result}
Let $G \subseteq {[n] \choose 2}$ be a graph on the vertex set $[n]$ and let $\Pi$ be a set
partition of $[n]$.  The arrangements $\Shi(G)$ and $\Ish(G)$ have the same number of 
total regions with ceiling partition $\Pi$ and dominant regions with ceiling partition $\Pi$.
\end{theorem}

It may be tempting to guess that Theorem~\ref{dominance-result} generalizes to 
state that for any permutation 
$\pi \in \symm_n$ and any 
set partition $\Pi$ of $[n]$, the arrangement $\Shi(G)$ and $\Ish(G)$ have the same number of 
regions $R$ on which the coordinate inequalities $x_{\pi_1} > \dots > x_{\pi_n}$ are satisfied
and have ceiling partition $\Pi$.  However, this generalization is false.  Indeed, there are $2$ regions
of $\Shi(3)$ and $3$ regions of $\Ish(3)$ which satisfy $x_1 > x_3 > x_2$.

Theorem~\ref{dominance-result} was proven by labeling the regions of $\Ish(G)$ with
certain combinatorial objects called ``Ish ceiling diagrams" and relating this labeling to 
(a slight modification of) a labeling of $\Shi(G)$ due to Athanasiadis and Linusson \cite{ALShi}.
Using these labels, one derives an explicit product formula for the number of regions of
$\Shi(G)$ or $\Ish(G)$ with a fixed ceiling partition.  While enumerative, this proof is not bijective.

\begin{problem}
\label{dominance-problem} (Armstrong-Rhoades)
Let $G \subseteq {[n] \choose 2}$ be a graph on the vertex set $[n]$.  Give a bijection between
the regions of $\Shi(G)$ and $\Ish(G)$ which preserves ceiling partitions and dominance.
\end{problem}

By combining degrees of freedom and ceiling partitions, one obtains a different equidistribution result.

\begin{theorem} (Armstrong-Rhoades)
\label{freedom-result}
Let $G \subseteq {[n] \choose 2}$ be a graph on the vertex set $[n]$ and let $\Pi$ be a set
partition of $[n]$.  Fix $d \geq 1$.  The arrangements $\Shi(G)$ and $\Ish(G)$ have the same number
of regions with ceiling partition $\Pi$ and $d$ degrees of freedom.  
\end{theorem}

The reader may wonder whether the
statements of Theorems~\ref{dominance-result} and \ref{freedom-result}
could be unified to 
state that $\Shi(G)$ and $\Ish(G)$ have the same number of dominant regions with a fixed
ceiling partition and degrees of freedom.  However, this statement is also false.  For example,
there are $2$ dominant regions of $\Shi(3)$ and $3$ dominant regions of $\Ish(3)$ with
$1$ degree of freedom.
As with Theorem~\ref{dominance-result}, the proof of Theorem~\ref{freedom-result}
is enumerative, but not bijective.

\begin{problem}
\label{freedom-problem}
(Armstrong-Rhoades)
Let $G \subseteq {[n] \choose 2}$ be a graph on the vertex set $[n]$.  Give a bijection between
the regions of $\Shi(G)$ and $\Ish(G)$ which preserves ceiling partitions and degrees
of freedom.
\end{problem}

In this paper we will give bijections which solve Problems~\ref{basic-problem}, 
\ref{dominance-problem}, and \ref{freedom-problem}.  
The maps which accomplish this are given in the following (noncommutative) diagram,
where all arrows are bijections (and we postpone the definitions of the maps and
intermediate objects until later sections). 
The middle composition $\omega \circ \alpha \circ \rho$ will solve Problem~\ref{basic-problem},
the bottom composition $\omega \circ \beta \circ \lambda \circ \widehat{\rho}$ will
induce a bijection (for an arbitrary subgraph $G$) solving Problem~\ref{dominance-problem}, and
the top composition $\omega \circ \gamma$ will induce a bijection solving Problem~\ref{freedom-problem}.
(Of course, a bijection solving either Problem~\ref{dominance-problem} or 
Problem~\ref{freedom-problem} automatically solves Problem~\ref{basic-problem}.) 

\begin{center}
\includegraphics[scale = 0.5]{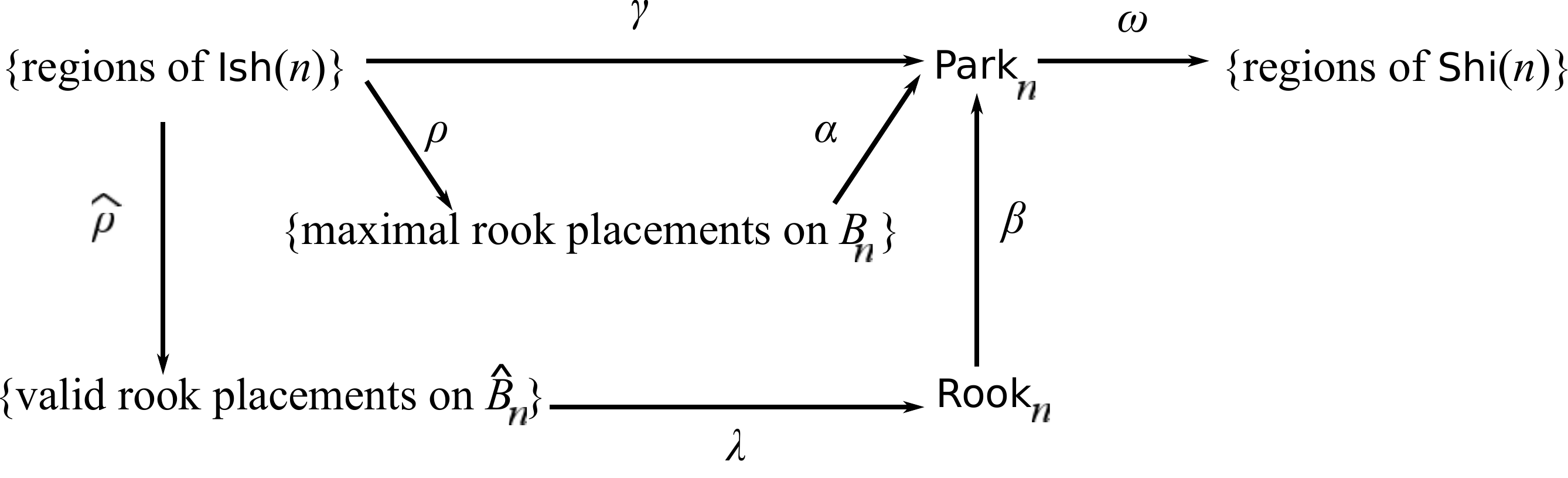}
\end{center}

The map $\omega: \Park_n \xrightarrow{\sim} \{$regions of $\Shi(n)\}$ is (the inverse of) a labeling of 
regions of the Shi arrangement by parking functions.  This is a minor modification of a labeling due
to Athanasiadis and Linusson \cite{ALShi}.

The maps $\rho$ and $\widehat{\rho}$ give closely related ways to encode regions of 
$\Ish(n)$ in terms of rook placements on certain boards $B_n$ and 
$\widehat{B}_n$.  
For any subgraph $G \subseteq {[n] \choose 2}$, there are ``restrictions" of $\rho$ and $\widehat{\rho}$
which encode the regions of $\Ish(G)$ as rook placements on certain
subboards $B_n(G)$ and $\widehat{B}_n(G)$.
 When the subboard $B_n(G)$
is rook-equivalent to a Ferrers board, one can use a product formula
of Goldman, Joichi, and White \cite{GJW} to obtain an expression for the number 
of regions of $\Ish(G)$.

The map $\alpha$ associates a parking function to any maximal rook placement on $B_n$.
While this map is simple to define, it is the most destructive bijection in the diagram and 
does not preserve any of the statistics of interest.

The set $\Rook_n$ of ``rook words of size $n$" is our new Ish analog of parking functions.
The composition $\lambda \circ \widehat{\rho}$ labels Ish regions with rook words just as 
$\omega^{-1}$ labels Shi regions with parking functions.  We will see that every orbit in the 
action of $\ZZ_{n+1}$ on the set of words $[n+1]^n$ contains a unique rook word and
a unique parking function.
``Cycle lemma" results of this kind are ubiquitous in combinatorics.  We have a 
canonical bijection $\beta: \Rook_n \rightarrow \Park_n$ obtained by sending a rook word
to the unique parking function in its orbit.

In addition to solving Problem~\ref{dominance-problem},
our cycle lemma  method  gives combinatorial intuition as to why the 
Shi/Ish duality $x_i - x_j = 1 \longleftrightarrow x_1 - x_j = i$ preserves so much information.  Namely,
we have the following heuristic.

\begin{quote}
``Shi/Ish duality consists of making two choices of distinguished orbit representatives in
the action of $\ZZ_{n+1}$ on $[n+1]^n$."
\end{quote}

The map $\gamma$ which solves Problem~\ref{freedom-problem} is the most 
complicated bijection in the diagram and labels Ish regions by parking functions, thought
of as labeled Dyck paths.

The remainder of the paper is organized as follows.  
In {\bf Section~\ref{Background}} we recall basic definitions related to words and set partitions
and explain how to label Shi and Ish regions by Shi and Ish ceiling diagrams.
In {\bf Section~\ref{Parking functions and rook words}} we recall the Athanasiadis-Linusson
labeling $\omega^{-1}$ of regions of $\Shi(n)$ by parking functions, introduce rook words, and give
the Cycle Lemma bijection $\beta: \Rook_n \rightarrow \Park_n$.
In {\bf Section~\ref{Rook placements and the Ish arrangement}}
we describe the maps $\rho$ and $\widehat{\rho}$ which label regions of the Ish arrangement
by rook placements.
In {\bf Section~\ref{A bijection between the regions}} we 
describe the map $\alpha$ which completes the bijection
$\omega \circ \alpha \circ \rho$ that solves Problem~\ref{basic-problem}.
In {\bf Section~\ref{Dominance section}},
we describe the map $\lambda$ that completes the bijection
$\omega \circ \beta \circ \lambda \circ \widehat{\rho}$ which solves Problem~\ref{dominance-problem}.
In {\bf Section~\ref{Bounded section}}, we
 use a prime
version of our Cycle Lemma to give a bijection
between the relatively bounded regions of Shi and Ish which preserves
ceiling partitions.
In {\bf Section~\ref{Freedom section}},
we describe the map $\gamma$ that completes the bijection $\omega \circ \gamma$ 
which solves Problem~\ref{freedom-problem}.
We close in {\bf Section~\ref{Closing remarks}} with a proposed improvement
on our solution to Problem~\ref{freedom-problem}.

\section{Background}
\label{Background}

\begin{figure}
\centering
\includegraphics[scale = 0.6]{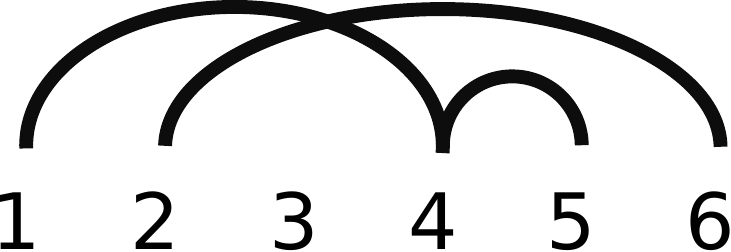}
\caption{The arc diagram of a set partition of $[6]$.}
\label{setpartition}
\end{figure}

\subsection{Set partitions and arc diagrams}

Given a set partition $\Pi$ of $[n]$, the {\sf arc diagram} of $\Pi$ is obtained by 
drawing the numbers $1, 2, \dots, n$ in a line and drawing an arc between $i$ and $j$
if $i < j$ and there exists a block $B$ of $\Pi$ such that $i$ and $j$ are consecutive elements of $B$.
The arc diagram of the set partition $\{ \{1, 4, 5\}, \{ 2, 6 \}, \{ 3 \} \}$ of $[6]$ is shown in Figure~\ref{setpartition}.
We will refer to arcs in the arc diagram of a set partition $\Pi$ as the ``arcs of $\Pi$".
If $\Pi$ is a set partition of $[n]$ and $\pi \in \symm_n$ is a permutation,
we let $\pi(\Pi)$ denote the set partition of $[n]$ whose blocks are
$\{ \pi(B) \,:\, B \in \Pi\}$.

A set partition $\Pi$ of $[n]$ is called {\sf nonnesting} if there do not exist 
indices $1 \leq a < b < c < d \leq n$ such that $a - d$ and $b - c$ are both arcs 
of $\Pi$.  
That is, the partition $\Pi$ is nonnesting if and only if there are no pairs of nested arcs of $\Pi$.
The set partition in Figure~\ref{setpartition}  is not nonnesting, but the set partition
$\{ \{1, 3\},  \{2, 4\} \}$ of $[4]$ is nonnesting.
There are $\Cat(n)$ nonnesting partitions of $[n]$.

To study the degrees of freedom statistic, it will be useful to break up set partitions 
into connected components.  
If $\Pi$ is a set partition of $[n]$, there exists a unique finest partition of $[n]$ of the form
$\Sigma = \{ \{1, 2, \dots, i_1\}, \{i_1 + 1, i_1 + 2, \dots, i_2 \}, \dots, \{i_{d-1} + 1, i_{d-1} + 2, \dots, n \} \}$
(for some indices $1 \leq i_1 < i_2 < \dots < i_{d-1} < n$) such that $\Pi$ refines $\Sigma$.
We call the restrictions of $\Pi$ to the blocks of $\Sigma$ the {\sf connected components} 
of $\Pi$.
The partition $\Pi$ is said to be {\sf connected} if it has a single connected component.
For example, the partition shown in Figure~\ref{setpartition} is connected and the set partition
$\{ \{1, 3\}, \{2 \}, \{4, 5, 6 \}, \{7 \} \}$ has $3$ connected components.

\subsection{Shi ceiling diagrams}

\begin{figure}
\centering
\includegraphics[scale = 0.6]{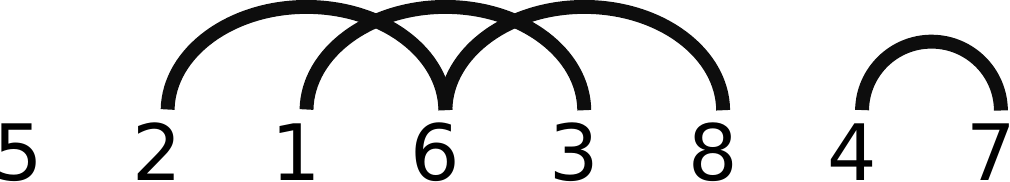}
\caption{A Shi ceiling diagram.}
\label{shidiagram}
\end{figure}

Let $G \subseteq {[n] \choose 2}$ be a graph.
We recall a labeling of the regions of $\Shi(G)$ appearing in \cite{AR} 
(which is a minor modification of a labeling of the regions of 
$\Shi(G)$ due to Athanasiadis and Linusson \cite{ALShi}).

Let $R$ be a region of $\Shi(G)$.  We associate to $R$ a pair $(\pi, \Pi)$ as follows, where
$\pi \in \symm_n$ is a permutation and $\Pi$ is a nonnesting partition of $[n]$.
We let $\pi = \pi_1 \dots \pi_n$ be the unique element of $\symm_n$ such that 
the coordinate inequalities $x_{\pi_1} > \dots > x_{\pi_n}$ hold on $R$.  We let $\Pi$ be 
the set partition of $[n]$ generated by $i \sim j$ whenever $x_{\pi_i} - x_{\pi_j} = 1$
is a ceiling of the region $R$.  The partition $\Pi$ is necessarily nonnesting.
We call the pair $(\pi, \Pi)$ the {\sf Shi ceiling diagram} of $R$ and visualize it by writing the 
one-line notation $\pi_1 \dots \pi_n$ from left to right and drawing an arc between 
$\pi_i$ and $\pi_j$ whenever $i \sim j$ in $\Pi$.

Figure~\ref{shidiagram} shows an example of a Shi ceiling diagram when $n = 8$.  The pair 
$(\pi, \Pi)$ is given by $\pi = 52163847 \in \symm_8$ and
$\Pi = \{ \{1\}, \{2, 4, 6\}, \{3, 5\}, \{7, 8\} \}$.  The ceilings of the region $R$ corresponding to
this diagram are $x_2 - x_6 = 1$, $x_1 - x_3  = 1$, $x_6 - x_8 = 1$, and $x_4 - x_7 = 1$.
The ceiling partition of $R$ is therefore
$\{ \{1, 3\}, \{2, 6, 8\}, \{4, 7\}, \{5\} \}$.  Observe that this ceiling partition is the image of $\pi(\Pi)$ of the 
set partition $\Pi$ under the permutation $\pi$.  The left of Figure~\ref{shiishthree} shows
the regions of the full Shi arrangement $\Shi(3)$ labeled by their Shi ceiling diagrams.

The following proposition characterizes which pairs $(\pi, \Pi)$ arise as ceiling diagrams for Shi regions.

\begin{proposition}
\label{shi-ceiling-diagram-characterization}
(Armstrong-Rhoades, Athanasiadis-Linusson)  The map $R \mapsto (\pi, \Pi)$ associating a region
of $\Shi(G)$ to its Shi ceiling diagram bijects regions of $\Shi(G)$ with pairs 
$(\pi, \Pi)$ where
\begin{itemize}
\item $\pi$ is a permutation in $\symm_n$ and $\Pi$ is a nonnesting set partition of $[n]$,
\item for every block $B = \{ b_1 < \dots < b_k \}$ of $\Pi$ we have
$\pi_{b_1} < \dots < \pi_{b_k}$, and
\item for every block $B = \{ b_1 < \dots < b_k \}$ of $\Pi$ and every $1 \leq i \leq k-1$, we have
that $(\pi_{b_i} < \pi_{b_{i+1}})$ is an edge in $G$.
\end{itemize}
Moreover, if $(\pi, \Pi)$ is the Shi ceiling diagram of a region $R$ of $\Shi(G)$ and $\Pi$ 
has $d$ connected components, then $R$ has $d$ degrees of freedom.  Also, the 
ceiling partition of $R$ is the image $\pi(\Pi)$ of the partition $\Pi$ under the permutation $\pi$.
\end{proposition}

For example, the Shi ceiling diagram in Figure~\ref{shidiagram} labels a region of 
$\Shi(G)$ if and only if $G$ contains the edges $1 < 3$, $2 < 6$, $4 < 7$, and $6 < 8$.  In this case,
the region labeled by this diagram has $3$ degrees of freedom.

Athanasiadis and Linusson considered a version of Shi ceiling diagrams where ``floors" are used
instead of ceilings.  The advantage of our labeling is that it makes the degrees of freedom 
statistic more visible.  In what follows, we will identify Shi regions with their Shi ceiling diagrams.
This identification gives rise to an embedding 
$\{$regions of $\Shi(G)\} \hookrightarrow \{$regions of $\Shi(n) \}$.  We will consider regions 
of $\Shi(G)$ as regions of $\Shi(n)$ without comment.

\subsection{Ish ceiling diagrams}  
Let $G \subseteq {[n] \choose 2}$ be a graph.  While regions 
of $\Shi(G)$ are labeled by permutations decorated with arcs, regions of $\Ish(G)$ are labeled
by permutations decorated with dots.

\begin{figure}
\centering
\includegraphics[scale = 0.6]{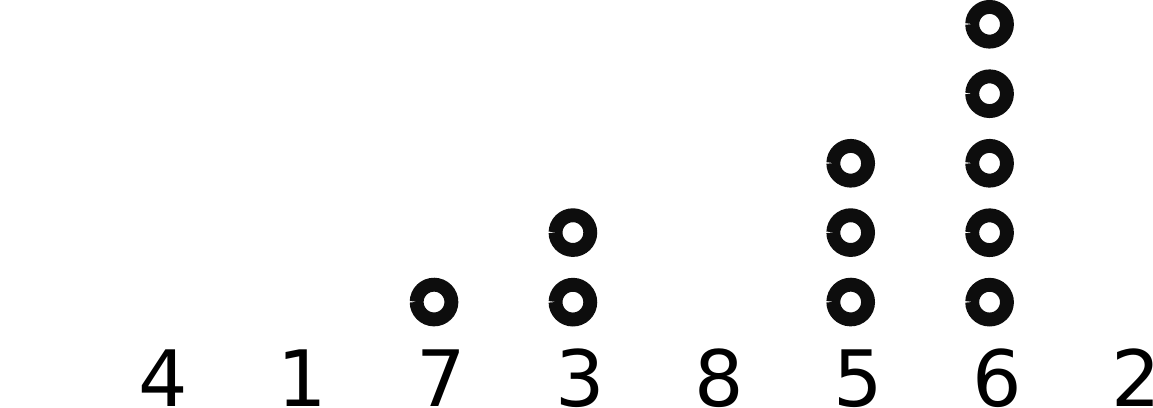}
\caption{An Ish ceiling diagram.}
\label{ishdiagram}
\end{figure}

Let $R$ be a region of $\Ish(G)$.  We associate to $R$ a pair $(\pi, \epsilon)$ as follows,
where $\pi \in \symm_n$ is a permutation and $\epsilon \in \ZZ_{\geq 0}^n$ is a length $n$ sequence
of nonnegative integers.  We let $\pi = \pi_1 \dots \pi_n$ be the unique permutation such that
the coordinate inequalities $x_{\pi_1} > \dots > x_{\pi_n}$ hold on $R$.  We define 
$\epsilon = \epsilon_1 \dots \epsilon_n$ by $\epsilon_j = i$ if $x_1 - x_{\pi_j} = i$ is a ceiling
of the region $R$  and $\epsilon_j = 0$ otherwise.  The pair $(\pi, \epsilon)$ is the 
{\sf Ish ceiling diagram} of $R$ and is visualized by drawing the one-line notation
$\pi = \pi_1 \dots \pi_n$ and placing $\epsilon_j$ dots on top of $\pi_j$ for all $1 \leq j \leq n$.

Figure~\ref{ishdiagram} gives an example of an Ish ceiling diagram in the case $n = 8$.  
We have that $(\pi, \epsilon)$ is given by
$\pi = 41738562 \in \symm_8$ and $\epsilon = 00120350$.  The ceilings of the region $R$ corresponding
to this diagram are $x_1 - x_7 = 1$, $x_1 - x_3 = 2$, $x_1 - x_5 = 3$, and $x_1 - x_6 = 5$.
The ceiling partition of $R$ is the set partition 
$\{ \{1, 7 \}, \{2, 3, 5, 6\}, \{4\}, \{8\} \}$.  The right of Figure~\ref{shiishthree} shows the regions of the
full Ish arrangement $\Ish(3)$ labeled by their Ish ceiling diagrams.

The Ish analog of Proposition~\ref{shi-ceiling-diagram-characterization} is as follows.  Roughly 
speaking, regions of $\Ish(G)$ are labeled by permutations $\pi = \pi_1 \dots \pi_n$ 
decorated by dot accumulation sequences
$\epsilon = \epsilon_1 \dots \epsilon_n$ such that at most $i-1$ dots are on top of $i$,
all positive dot accumulations appear to the right of $1$ and are strictly increasing from
left to right, and whenever we have  a positive dot accumulation of $i$ dots on top of $j$, 
we have that $(i < j)$ is an edge in $G$.

\begin{proposition}
\label{ish-ceiling-diagram-characterization} (Armstrong-Rhoades)
The map $R \mapsto (\pi, \epsilon)$ associating a region of $\Ish(G)$ to its Ish ceiling diagram
bijects regions of $\Ish(G)$ with pairs $(\pi, \epsilon)$ where
\begin{itemize}
\item $\pi$ is a permutation in $\symm_n$ and $\epsilon$ is a length $n$ sequence of nonnegative
integers,
\item if $i < j$ and $\epsilon_i, \epsilon_j > 0$, then $\epsilon_i < \epsilon_j$,
\item if $\epsilon_i > 0$, then $\pi^{-1}_1 < i$,
\item $\epsilon_i < \pi_i$ for $1 \leq i \leq n$, and
\item if $\epsilon_i > 0$, we have that $(\epsilon_i < \pi_i)$ is an edge in $G$.
\end{itemize}
Moreover, if $(\pi, \epsilon)$ is the Ish ceiling diagram associated to a region $R$ of
$\Ish(G)$ and $k$ is the largest index $1 \leq k \leq n$ such that $\epsilon_k > 0$,
then $R$ has $n - k + \pi^{-1}_1$ degrees of freedom (where we take $k = \pi^{-1}_1$ if $\epsilon$
is the zero sequence).  Also, the ceiling partition of $R$ is the set partition of $[n]$ generated
by $\epsilon_i \sim \pi_i$, where $1 \leq i \leq n$ is such that $\epsilon_i > 0$.
\end{proposition}

For example, the Ish ceiling diagram of Figure~\ref{ishdiagram} labels a region of 
$\Ish(G)$ if and only if $G$ contains the edges $1 < 7$, $2 < 3$, $3 < 5$, and $5 < 6$.  Assuming
$G$ contains these edges, the number of degrees of freedom is 
$8 - 7 + 2 = 3$.  Note that we could also calculate degrees of freedom by merging 
the letters of $\pi$ between $1$ and the rightmost dotted letter into a single symbol $\star$
and counting the number of objects in the resulting sequence.  In this example, we
obtain the sequence $4 \star 2$, which contains $3$ objects as desired.

\subsection{Words}  Given a word $w = w_1 w_2 \dots w_k$ (with letters in some alphabet),
we obtain a set partition $\Pi(w)$ of $[k]$ by the rule $i \sim j$ if and only if $w_i = w_j$.  
For example, we have that $\Pi(1331) = \{ \{ 1, 4 \}, \{2, 3\} \}$.  We call $\Pi(w)$ the {\sf position partition}
of the word $w$.

Given $m, k > 0$, the set $[m]^k$ of length $k$ words in the alphabet $[m]$ carries
a free action of the cyclic group $\ZZ_m$ generated by
$w_1 \dots w_k \mapsto (w_1 + 1) \dots (w_k + 1)$, where letters are interpreted modulo $m$.  
This action preserves position partitions and partitions $[m]^k$ into $m^{k-1}$ orbits, each of size $m$.
In the case $m = n+1$ and $k = n$ we will study two distinguished choices of orbit representatives
called parking functions and rook words.  
In the case $m = n-1$ and $k = n$ will study two other distinguished choices of orbit representatives
called prime parking functions and prime rook words.

\section{Parking functions and rook words}
\label{Parking functions and rook words}

\subsection{Parking functions and the Shi arrangement}
A word $w_1 \dots w_n$ with letters in $[n]$
is called a {\sf parking function of size $n$} if the nondecreasing
rearrangement $a_1 \leq \dots \leq a_n$ of the letters in $w_1 \dots w_n$
satisfies $a_i \leq i$ for all $i$.
\footnote{This terminology arises from the following situation.  
Consider a linear parking lot consisting of $n$ spaces and $n$ cars which wish to park
in the lot.  Car $i$ prefers to park in the spot $w_i$.  At stage $i$ of the parking process,
car $i$ will park in the first available spot $\geq w_i$ if any such spots are available.  Otherwise,
car $i$ leaves the lot.  The preference sequence $w_1 \dots w_n$ is a parking function if and only if 
every driver can park.}
We denote by $\Park_n$ the set of parking functions of size $n$.
For example, we have that 
\begin{equation*}
\Park_3 = \{111, 112, 121, 211, 113, 131, 311, 122, 212, 221, 123, 213, 132, 231, 312, 321\}.
\end{equation*}

Parking functions were introduced by Konheim and Weiss \cite{KW} in the context of a hashing problem
in computer science, but have since received a great deal of attention in algebraic combinatorics.
Most famously, Haiman \cite{Haiman} proved that the action of the symmetric group $\symm_n$  on $\Park_n$
given by
\begin{equation*}
\pi.(w_1 \dots w_n) := w_{\pi_1} \dots w_{\pi_n}
\end{equation*}
is isomorphic to the action of $\symm_n$ on the space $DH_n$ of diagonal harmonics.

The number of parking functions of size $n$ is  given by $|\Park_n| = (n+1)^{n-1}$.  
After Shi proved that $\Shi(n)$ has $(n+1)^{n-1}$ regions, it became natural to ask for a labeling 
of the regions of $\Shi(n)$ by parking functions.  The first such labeling was given by a recursive procedure
of Pak and Stanley \cite{Stanley} where one labels the `fundamental alcove' defined by
$x_1 > x_2 \dots > x_n$ and $x_1 - x_n < 1$ by the parking function
$1 1 \dots 1$ and modifies this parking function appropriately as one crosses the hyperplanes 
of $\Shi(n)$.
Athanasiadis and Linusson \cite{ALShi} gave a more direct bijection as follows.

Let $G \subseteq {[n] \choose 2}$ be a graph.
Given a region $R$ of $\Shi(G)$ with Shi ceiling diagram $(\pi, \Pi)$, we will obtain a parking 
function $w_1 \dots w_n \in \Park_n$.  For $1 \leq i \leq n$, suppose that 
$i$ belongs to the block $\{\pi_{b_1} < \dots < \pi_{b_k} \}$ of the set partition $\pi(\Pi)$ 
(where $\{b_1 < \dots < b_k\}$ is a block of $\Pi$).
We let $w_i = b_1$.  In other words, we let $w_i$ be the minimal position in the block
of $\Pi$ which is decorated by the letter $i$ in the permutation $\pi$.  
For example, if $R$ is the region whose Shi ceiling diagram is given in Figure~\ref{shidiagram},
the associated parking function is $32371272 \in \Park_8$.
The next proposition
is essentially due to Athanasiadis and Linusson.

\begin{proposition}
\label{parking-shi-characterization}
Let $G \subseteq {[n] \choose 2}$ be a graph.
The procedure in the last paragraph bijects regions $R$ of $\Shi(G)$ with parking functions 
$w_1 \dots w_n \in \Park_n$ such that every arc in the position partition of the word
$w_1 \dots w_n$ is an edge in the graph $G$.   In particular, the ceiling partition of $R$
equals the position partition of $w_1 \dots w_n$.
\end{proposition}

We denote by $\omega$ the bijection $\Park_n \xrightarrow{\sim} \{$regions of $\Shi(n) \}$ given by 
Proposition~\ref{parking-shi-characterization}.
For any graph $G \subseteq {[n] \choose 2}$, the bijection $\omega$ restricts to a bijection from the set of parking
functions $w_1 \dots w_n$ such that the arcs of the position partition of $w_1 \dots w_n$ are contained in $G$
to the set of regions of $\Shi(G)$.

\subsection{Rook words}
We are ready to define our new Ish analog of parking functions.
A word $w = w_1 \dots w_n$ with letters in $[n]$ is called a {\sf rook word} if all of the integers
in the closed interval $[1, w_1]$ appear among the letters of $w$.
\footnote{This terminology will soon be justified.}
For example, the word $31552$ is a rook word but $24453$ and $31111$ are not.  We denote
by $\Rook_n$ the set of rook words of size $n$.  For example,
\begin{equation*}
\Rook_3 = \{111, 112, 121, 211, 113, 131, 133, 122, 212, 221, 123, 213, 132, 231, 312, 321 \}.
\end{equation*}

Observe that words can be both parking functions and rook words;  indeed, the difference
$\Park_3 - \Rook_3$ contains the single element $311$ and
$\Rook_3 - \Park_3$ contains the single element $133$.   
The authors do not have a conjecture for the cardinality of $\Park_n \cap \Rook_n$.
As is typical for Ish combinatorics, the 
defining condition of rook words treats the first letter $w_1$ differently from the other letters.
As a result, rook words are less ``symmetric" than parking functions:  the set $\Rook_n$ 
does not carry an action of the full symmetric group $\symm_n$ by letter permutation, but only the parabolic
subgroup $\symm_1 \times \symm_{n-1}$ 
of permutations $\pi \in \symm_n$ satisfying $\pi_1 = 1$.  

\subsection{The Cycle Lemma}
Rook words will ultimately be used to label regions of the Ish arrangement.  By 
Proposition~\ref{parking-shi-characterization}, to relate the regions of Shi and Ish, we
will want to relate parking functions and rook words.  
This will be accomplished by the following Cycle Lemma (which also shows that 
$|\Rook_n| = (n+1)^{n-1}$).

Cycle lemmas play an important role in enumerative combinatorics.  
When one wants to enumerate a finite set $S$ of combinatorial objects, one finds a superset $T \supseteq S$
carrying a free action of a  cyclic group $C$ such that every orbit in the action of $C$ on $T$ contains
a unique element of $S$.  
This reduces the problem of finding $|S|$ to the problem of finding $|T|$ and $|C|$.
Among other things, this basic technique can be used to count 
Dyck paths of size $n$, their Narayana and Kreweras refinements, and their Fuss and rational generalizations.

\begin{lemma} 
\label{cycle}
(Cycle Lemma) Every orbit of the action of 
$\ZZ_{n+1}$ on $[n+1]^n$ contains a unique parking function and a unique rook word, so that the number
of parking functions or rook words is $(n+1)^{n-1}$.
\end{lemma}

\begin{proof}
Let $w = w_1 \dots w_n$ be a word in $[n+1]^n$.  Pollak proved that the $\ZZ_{n+1}$-orbit of $w$ 
contains a unique parking function using the following beautiful argument.  Consider a circular parking
lot with $n+1$ parking spaces circularly labeled with $1, 2, \dots, n+1$
and $n$ cars which want to park in the lot.  Interpret $w = w_1 \dots w_n$
as a driver preference sequence.  Since the lot is circular, the parking process will be successful and result
in a single empty space.  The preference sequence $w$ is a parking function if and only if this empty space
is the `additional' space $n+1$ and the action of the cyclic group $\ZZ_{n+1}$ rotates this empty space.

We want to show that the $\ZZ_{n+1}$-orbit of $w$ also contains a unique rook word.  To do this, choose
$w' = w_1' \dots w_n'$ in the $\ZZ_{n+1}$-orbit of $w$ with $w_1'$ maximal such that 
every integer in the interval $[1, w_1']$ appears as a letter of $w'$.  The maximality of $w_1'$ implies that
$n+1$ does not appear among the letters of $w'$, so $w'$ is a rook word.  It is clear that $w'$ is also the 
unique rook word in the $\ZZ_{n+1}$-orbit of $w$.
\end{proof}

For example, the unique parking function in the $\ZZ_6$-orbit
$\{14425, 25536, 36641, 41152, 52263, 63314 \}$ of $14425 \in \Rook_5$ is $41152 \in \Park_5$.
Lemma~\ref{cycle} gives a canonical bijection $\beta: \Rook_n \xrightarrow{\sim} \Park_n$
which preserves position partitions.
For example, we have $\beta(14425) = 41152$.
This bijection is at the heart of 
solving Problem~\ref{dominance-problem} and gives the following isomorphism of $\symm_{n-1}$-sets
(where we identify $\symm_{n-1} \cong \symm_1 \times \symm_{n-1}$):
\begin{equation}
\Park_n \downarrow^{\symm_n}_{\symm_{n-1}} \cong_{\symm_{n-1}} \Rook_n.
\end{equation}

We also have a `prime' version of the Cycle Lemma which keeps track of relatively bounded regions. 
We recall the notion of a prime parking function
and introduce the notion of a prime rook word.

A word $w = w_1 \dots w_n$ with letters in $[n-1]$ is called a {\sf prime parking function of size $n$}
if the nondecreasing rearrangement $a_1 \leq \dots \leq a_n$ of $w$  satisfies
the inequalities $a_1 \leq 1$,
$a_2 \leq 1$, $a_3 \leq 2, \dots, a_n \leq n-1$.  
We denote by $\Park_n'$ the set of prime parking functions of size $n$.
For example, we have
 $\Park_3' = \{111, 112, 121, 211 \}$.

A word $w = w_1 \dots w_n$ with letters in $[n-1]$ is called a {\sf prime rook word of size $n$} if 
$w_1 = 1$.  A prime rook word is automatically a rook word.  
We denote by $\Rook'_n$ the set of prime rook words of size $n$.  
For example, we have
$\Rook_3' = \{111, 112, 121, 122\}$.

The following prime version of the 
Cycle Lemma is well known in the case of prime parking functions and trivial in the case of 
prime rook words.

\begin{lemma}
\label{cycle-prime} (Prime Cycle Lemma)
Every orbit in the action of $\ZZ_{n-1}$ on $[n-1]^n$ contains a unique prime parking function 
and a unique prime rook word, so that the number of prime parking functions or prime 
rook words is $(n-1)^{n-1}$.
\end{lemma}

We denote by $\beta': \Rook_n' \rightarrow \Park_n'$ the bijection induced by Lemma~\ref{cycle-prime}.
For example, we have $\beta'(122) = 211$.  As with $\beta$, the bijection $\beta'$ preserves 
position partitions.

\section{Rook placements and the Ish arrangement}
\label{Rook placements and the Ish arrangement}

\subsection{Rook placements}  In order to give our bijections between the
regions of the Shi and Ish arrangements, we will reinterpret Ish ceiling 
diagrams in terms of rook placements.

A {\sf board} $B$ is a finite subset of the integer lattice $\ZZ \times \ZZ$.  A {\sf rook placement} on $B$ 
is a placement of non-attacking rooks on the squares of $B$.  
A rook placement on $B$ is called {\sf maximal} if it is maximal among the collection of all rook placements,
where we identify rook placements with subsets of $B$.
We will focus on rook placements
on two families of boards.  

For $n \geq 1$ we let $\widehat{B}_n$ be the ``bottom justified" board with 
column heights from left to right given by $(n, n+1, \dots, 2n)$ and $B_n$ be the bottom justified
board with left-to-right column heights $(n+1, n+2, \dots, 2n)$.
 We define the coordinates of $\widehat{B}_n$ and $B_n$ as follows:
\begin{align}
\widehat{B}_n &:= \{ (i, j) \,:\, 1 \leq i \leq n, 1 \leq j \leq n+i-1 \}. \\
B_n &:= \{ (i, j) \,:\, 2 \leq i \leq n, 1 \leq j \leq n+i-1 \} 
\end{align}
The boards $\widehat{B}_n$ and $B_n$ naturally decompose into a lower rectangular part 
below the line $y = n$ and an upper triangular part above the line $y = n$.  We will refer 
to the squares in these parts as being {\sf below the bar} or {\sf above the bar}, respectively.

The boards $\widehat{B}_8$ and $B_8$ are shown in Figure~\ref{ishrook}.  ``The bar" is bold and blue.
As shown in Figure~\ref{ishrook}, we label the columns of $\widehat{B}_n$ and $B_n$ by their $x$-coordinates
and the rows of $\widehat{B}_n$ and $B_n$ by either their $y$-coordinates or their $y$-coordinates minus $n$,
depending on whether these rows are above or below the bar.

If $G \subset { [n] \choose 2 }$ is a graph, we denote by $\widehat{B}_n(G)$ (resp. $B_n(G)$)
the subset of $\widehat{B}_n$ (resp. $B_n$) obtained by deleting the squares
$(j, n+i)$ above the bar for which $(i < j)$ is not an edge of $G$.  At the extremes,
if $G = K_n$ is the complete graph we have 
$\widehat{B}_n(K_n) = \widehat{B}_n$
and $B_n(K_n) = B_n$
 and if $G = \emptyset$ is the empty graph
 we have that $\widehat{B}_n(\emptyset)$ is the rectangular board
 $\{1, 2, \dots, n\} \times \{1, 2, \dots, n\}$ and
$B_n(\emptyset)$ is the rectangular board $\{2, 3, \dots, n\} \times \{1, 2, \dots, n\}$.

\subsection{Rook placements and Ish regions}

Let $G \subseteq {[n] \choose 2}$ be a graph.
Given a region $R$ of $\Ish(G)$ with Ish ceiling diagram
$(\pi, \epsilon)$, we define rook placements $\widehat{\rho}(R)$ and $\rho(R)$
on  $\widehat{B}_n$  and $B_n$ as follows.  Write $\pi = \pi_1 \dots \pi_n$ and
$\epsilon = \epsilon_1 \dots \epsilon_n$.
The rook placement $\widehat{\rho}(R)$ on $\widehat{B}_n$ consists of $n$ rooks, one in every column.
For $1 \leq i \leq n$, Rook $i$ is placed as follows.
\begin{itemize}
\item If $\epsilon_{\pi^{-1}_i} = 0$ (i.e., if there are no dots above $i$ in the Ish ceiling diagram of $R$),
then Rook $i$ is placed below the bar in position $(i, \pi^{-1}_i)$.
\item If $\epsilon_{\pi^{-1}_i} > 0$ (i.e., if there are dots above $i$ in the Ish ceiling diagram of $R$),
then Rook $i$ is placed above the bar in position $(i, n + \epsilon_i)$.
\end{itemize}
The rook placement $\rho(R)$ on $B_n$ is just the restriction of $\widehat{\rho}(R)$ from
$\widehat{B}_n$ to $B_n$ (so that
$\rho(R)$ contains $n-1$ rooks).  Either of two rook placements contains the same information as 
the Ish ceiling diagram.

The left of Figure~\ref{ishrook} shows the rook placement $\rho(R)$ for the region $R$
whose Ish ceiling diagram is shown in Figure~\ref{ishdiagram}.  The right of
Figure~\ref{ishrook} shows the rook placement $\widehat{\rho}(R)$.

\begin{figure}[H]
\centering
\includegraphics[scale = 1]{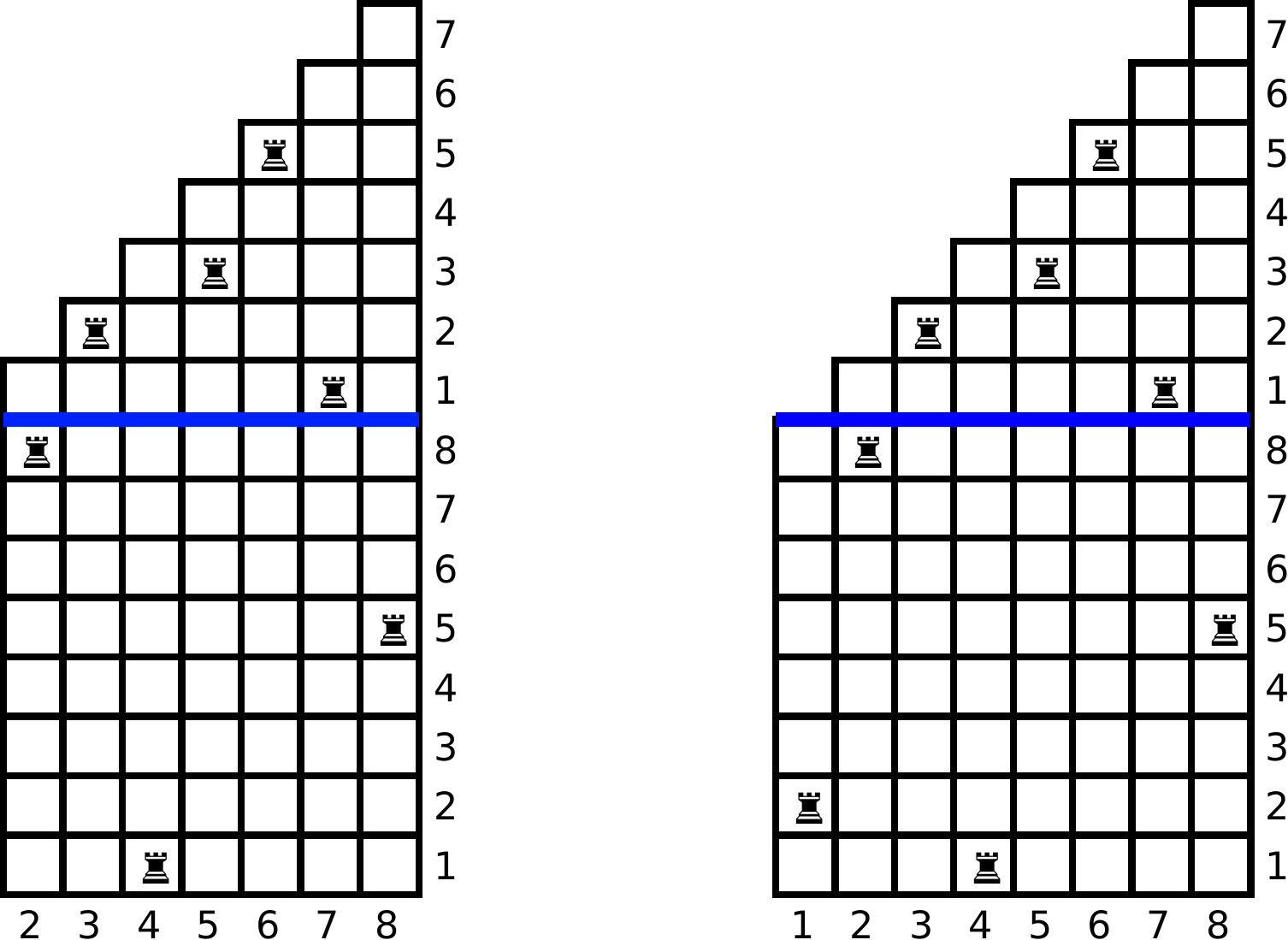}
\caption{The rook placements corresponding to the Ish region of Figure~\ref{ishdiagram}.}
\label{ishrook}
\end{figure}

\begin{lemma}
\label{same-information}
The rook placement $\widehat{\rho}(R)$ is determined by the rook placement $\rho(R)$.
The Ish ceiling diagram $(\pi, \epsilon)$ of $R$ is determined by either $\rho(R)$ 
or $\widehat{\rho}(R)$.
\end{lemma}

\begin{proof}
If $\epsilon$ is the zero vector (i.e., if $\pi$ has no dots above it), the fact that
$\rho(R)$ determines $\widehat{\rho}(R)$ 
is just the statement that $\pi^{-1}_1$ is determined by $\pi^{-1}_2, \dots, \pi^{-1}_n$.  Otherwise,
Proposition~\ref{ish-ceiling-diagram-characterization} states that 
$\pi^{-1}_1$ is less than every value of $j$ for which $\epsilon_j > 0$.  This means that
the unique rook in $\widehat{\rho}(R) - \rho(R)$ must be in the lowest row which is unoccupied
in $\rho(R)$.  

The rook placement $\widehat{\rho}(R)$ describes the positions and values of the undotted
letters in the permutation $\pi = \pi_1 \dots \pi_n$ as well as the values and number of dots above
the dotted letters of $\pi$.  Since positive accumulations of dots increase from left to right in an Ish 
ceiling diagram, this also determines the positions of the dotted letters of $\pi$.  We conclude
that $\widehat{\rho}(R)$ determines the full Ish ceiling diagram $(\pi, \epsilon)$.
\end{proof}

It will be useful to characterize the rook placements $\rho(R)$ and
$\widehat{\rho}(R)$ which come from regions 
$R$ of $\Ish(G)$.  The proof of the following reformulation of 
Proposition~\ref{ish-ceiling-diagram-characterization} is left to the reader.  We call a 
maximal rook placement on $\widehat{B}_n(G)$ {\sf valid} if every row below the rook in Column 1 contains
a rook.

\begin{lemma}
\label{rook-reformulation}
Let $G \subseteq {[n] \choose 2}$ be a graph.
The map $R \mapsto \rho(R)$ bijects regions of $\Ish(G)$ with maximal rook placements on $B_n(G)$.
The map $R \mapsto \widehat{\rho}(R)$ bijects regions of $\Ish(G)$ with valid rook
placements on $\widehat{B}_n(G)$.
\end{lemma}

Lemma~\ref{rook-reformulation}
implies that the number of regions of $\Ish(G)$ equals the number of 
maximal rook placements on $B_n(G)$.  In the special case
$G = K_n$, we have that $B_n(G) = B_n$ and we see immediately that $\Ish(n)$ has
$(n+1)^{n-1}$ regions (there are $(n+1)$ choices for the rook in the first column of $B_n$, then
$(n+2)-1$ choices for the rook in the second column, then $(n+3)-2$ choices for the 
rook in the third column, etc., leading to $n+1$ choices for each of the $n-1$ columns of $B_n$).  At 
the other extreme $G = \emptyset$, the obvious bijection between maximal rook placements 
on the below the bar portion of $B_n$ and permutation matrices recovers the fact that
 $\Cox(n)$ has $n!$ regions.
 
 More generally, for any graph $G \subseteq {[n] \choose 2}$ and $1 \leq k \leq n$, 
let $\Stir(G, k)$ denote the number of set partitions $\Pi$ of $[n]$ with exactly $k$ blocks such that 
every arc of $\Pi$ is an edge of $G$.  
When $G = K_n$ is the complete graph, the number $\Stir(K_n, k)$ is the classical Stirling number
of the second kind counting set partitions of $[n]$ with exactly $k$ blocks.  At the other extreme, we have
$\Stir(\emptyset, k) = \delta_{k, n}$.

A maximal rook placement on the board $B_n(G)$ consists of $n-1$ rooks.  If $k$ of these rooks 
are above the bar, then the rooks above the bar correspond to the arcs of a set partition counted
by $\Stir(G, n-k)$ (where an above the bar rook in position $(j, n+i)$ corresponds to the 
arc $(i < j)$).  If we fix the positions of the $k$ rooks above the bar, we are left with
$n (n-1) \cdots (k+2) = \frac{n!}{(k+1)!}$ choices for how to place the $n-k-1$ rooks below the bar.
Summing up over $k$, we have that the total number of maximal rook placements on 
$B_n(G)$ (or the number of regions of $\Ish(G)$) is
\begin{equation}
\label{region-enumeration}
\sum_{k = 0}^{n-1} \Stir(G, n-k) \frac{n!}{(k+1)!}.
\end{equation}
Using similar reasoning, one can show that if $\Pi$ is a partition of $[n]$ counted by 
$\Stir(G, n-k)$, then the number of regions of $\Ish(G)$ with ceiling partition $\Pi$ equals
\begin{equation}
\frac{n!}{(k+1)!}.
\end{equation}
This is equivalent to the Ish statement of
\cite[Theorem 5.1 (1)]{AR}.

Equation~\ref{region-enumeration} also follows from an expression for the 
characteristic polynomial $\chi_{\Ish(G)}(p) \in \ZZ[p]$ of the arrangement
$\Ish(G)$ contained in \cite{AR}.  Namely,
in \cite[Theorem 3.2]{AR} it is shown that
\begin{equation}
\label{ish-characteristic-polynomial}
\chi_{\Ish(G)}(p) = p \sum_{k=0}^{n-1} (-1)^k \Stir(G, n-k) (p-k-1) (p-k-2) \cdots (p-n+1).
\end{equation}
Evaluating Equation~\ref{ish-characteristic-polynomial} at $p = -1$ and applying Zaslavsky's Theorem
shows that the number of regions of $\Ish(G)$ is given by Equation~\ref{region-enumeration}.

Finally, if $B$ is any board and $m \geq 0$, the {\sf rook number} $r_m(B)$ counts the number 
of ways to place $m$ rooks on $B$.  By the same reasoning used to derive Equation~\ref{region-enumeration},
we have that
\begin{equation}
r_m(B_n(G)) = \sum_{k = 0}^{m} \Stir(G, n-k) {n-k-1 \choose m-k} \frac{n!}{(n-m-k-1)!},
\end{equation}
for $0 \leq m \leq n-1$.  The numbers $r_m(B_n(G))$ count equivalence classes of regions of
$\Ish(G)$ corresponding to ways to complete an $m$-rook placement to a maximal rook placement.
The role non-maximal rook placements play in Shi/Ish theory is unclear.

\section{A bijection between the regions of $\Shi(n)$ and $\Ish(n)$}
\label{A bijection between the regions}

\begin{figure}
\centering
\includegraphics[scale = 1]{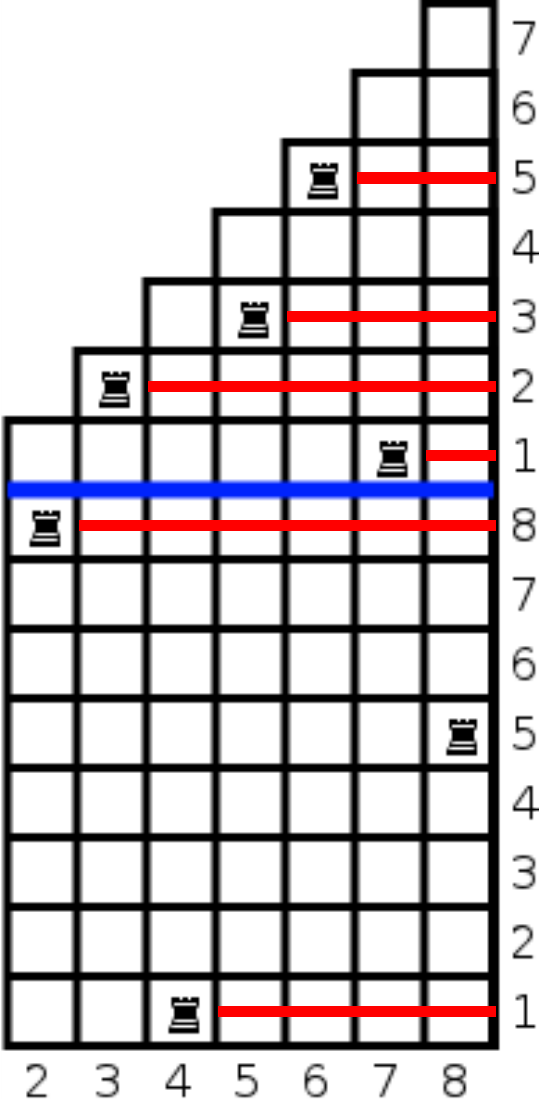}
\caption{The horizontal laser construction.}
\label{horizontalrook}
\end{figure}

We are ready to solve Problem~\ref{basic-problem} by giving a bijection between the regions of 
$\Shi(n)$ and $\Ish(n)$.  This will be the easiest of our bijections to state, but will not preserve 
dominance, ceiling partitions, or degrees of freedom and will not in general restrict to
a bijection between the regions of $\Shi(G)$ and $\Ish(G)$ for arbitrary subgroups $G$.

We define a map $\alpha: \{$maximal rook placements on $B_n \} \rightarrow \Park_n$ as follows.  If 
$\rho$ is a maximal rook placement on $B_n$, fire lasers rightward from every rook of $\rho$ to the right 
 side of the board $B_n$.  Figure~\ref{horizontalrook} shows an example of this construction.
 We define a word $v_1 \dots v_n \in [n+1]^n$ by the rule
\begin{equation}
v_i = \begin{cases}
1 & \text{if $i = 1$,} \\
\# \left\{  \begin{array}{c} \text{squares weakly below the rook in} \\ \text{Column $i$ which do not contain
laser fire} \end{array}
\right\} &
\text{if $2 \leq i \leq n$.}
\end{cases}
\end{equation}
 By the Cycle Lemma for parking functions, there exists a unique parking function
 $w_1 \dots w_n$ in the $\ZZ_{n+1}$-orbit of $v_1 \dots v_n$.  We set 
 $\alpha(\rho) := w_1 \dots w_n$.
 In the case of Figure~\ref{horizontalrook}, the word $v_1 \dots v_8$ is given by
 $18918974$ and $w_1 \dots w_n$ is given by $42342315$.
 
 \begin{theorem}
 \label{basic-theorem}
 The composition $\omega \circ \alpha \circ \rho$ is a bijection from the set of regions of 
 $\Ish(n)$ to the set of regions of $\Shi(n)$.
 \end{theorem}
 
 \begin{proof}
 By Proposition~\ref{parking-shi-characterization} and Lemma~\ref{rook-reformulation},
 it is enough to show that the map $\alpha$ is a bijection.  In terms of the intermediate words
 $v_1 \dots v_n$ used to define $\alpha(\rho)$ for a maximal rook placement $\rho$ on 
 $B_n$, this reduces to showing that every word in $[n+1]^{n-1}$ arises as 
 $v_2 \dots v_n$ for some rook placement $\rho$.  
 
 Given a word $v_2 \dots v_n \in [n+1]^{n-1}$, we construct this rook placement $\rho$ on $B_n$
 as follows.
 \begin{itemize}
 \item Place Rook $1$ in Column $2$ at height $v_2$ and fire a laser rightward from Rook $1$ to 
 the right side of $B_n$.
 \item For $1 \leq i < n-1$, having placed and fired Rooks $1, 2, \dots, i-1$, place Rook $i$ at the
 unique height in Column $i+1$ such that there are $v_{i+1}$ non-laser squares weakly below
 Rook $i$.
 \end{itemize}
 It can be checked that the `intermediate word' in the construction of $\alpha(\rho)$ is
 $1v_2 \dots v_n$, as desired.
 \end{proof}
 
 Continuing with our example,
 if $R$ is the region of $\Ish(8)$ whose Ish ceiling diagram is shown in Figure~\ref{ishdiagram},
 we calculated that
 $(\alpha \circ \rho)(R) = 42342315 \in \Park_8$.  We conclude that the composition
 $(\omega \circ \alpha \circ \rho)(R)$ is the region of $\Shi(8)$ whose Shi ceiling diagram is
 $(\pi, \Pi)$, where $\pi = 7 2 3 1 8 5 6 4$ and 
 $\Pi = \{ \{1\}, \{2, 6\}, \{3, 7\}, \{4, 8\}, \{5\} \}$. 
 The ceiling partition of this Shi region is
 $\pi(\Pi) = \{ \{7\}, \{2, 5\}, \{3, 6\}, \{1, 4\}, \{8\} \}$, which does not agree with the ceiling partition of the
 Ish region $R$.  Moreover, the region $R$ has $3$ degrees of freedom while the Shi region
 $(\omega \circ \alpha \circ \rho)(R)$ has $2$ degrees of freedom.

\section{A bijection between the regions of $\Shi(G)$ and $\Ish(G)$ preserving ceiling partitions
and dominance}
\label{Dominance section}

Let $G \subseteq {[n] \choose 2}$ be a graph.
By working a little harder than in the last section, we can get a bijection between regions of
$\Shi(G)$ and $\Ish(G)$ which preserves ceiling partitions and dominance.  
We again exploit Proposition~\ref{parking-shi-characterization} and Lemma~\ref{rook-reformulation}
to recast this problem in terms of rook placements. 

 Observe that any bijection between
 the regions of $\Shi(n)$ and $\Ish(n)$ which preserves ceiling partitions and dominance 
 restricts to a bijection between the regions of $\Shi(G)$ and $\Ish(G)$ which preserves ceiling 
 partitions and dominance.  Without loss of generality, we  restrict to the case $G = K_n$.

\begin{figure}
\centering
\includegraphics[scale = 1]{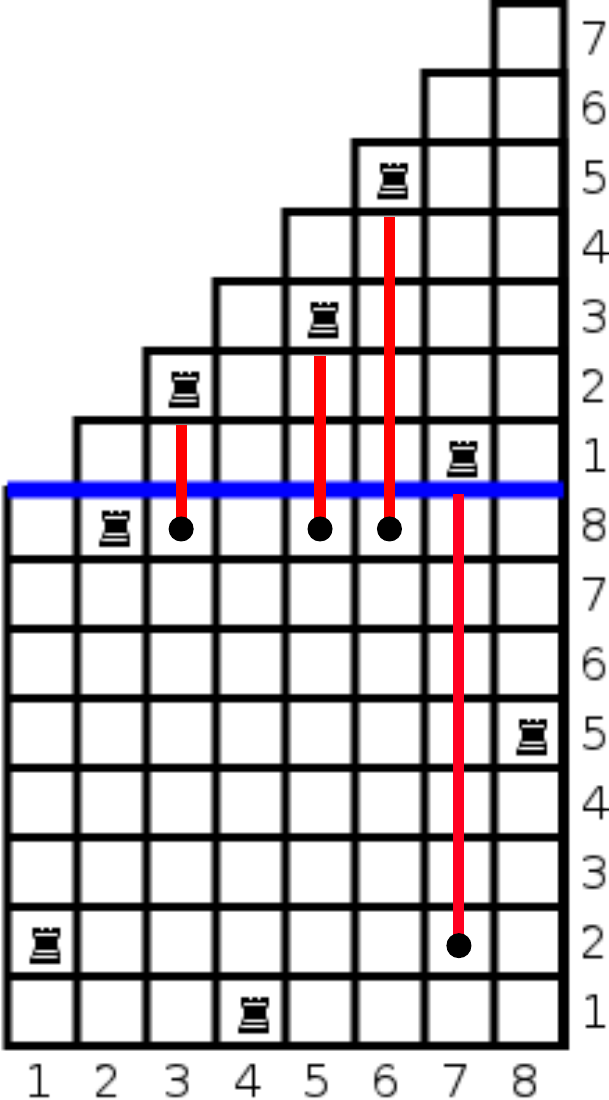}
\caption{The vertical laser construction.}
\label{verticalrook}
\end{figure}

We define a map $\lambda: \{$valid rook placements on $\widehat{B}_n\} \rightarrow \Rook_n$ 
as follows.  Let $\widehat{\rho}$ be a valid rook placement on $\widehat{B}_n$.
As with the map $\alpha$ of the last section, the first step is to fire lasers from the rooks 
of $\widehat{\rho}$.  In contrast to the map $\alpha$, the lasers defining $\lambda$ will be fired
down rather than to the right, will not in general reach the end of the board, and will only be
fired from rooks above the bar.  

More precisely, let Rooks $1, 2, \dots, n$ be the rooks in
Columns $1, 2, \dots, n$.  We read across $\widehat{B}_n$ from left to right.  If Rook $j$ is 
above the bar with coordinates $(j, n+i)$, fire a laser downwards from Rook $j$ with 
lower
endpoint given as follows (where we use the fact that $i < j$).
\begin{itemize}
\item If Rook $i$ is below the bar, the laser from Rook $j$ ends in the row of Rook $i$.
\item If Rook $i$ is above the bar, the laser from Rook $j$ ends in the same row as the laser fired down from
Rook $i$.
\end{itemize}
This vertical laser construction is shown in Figure~\ref{verticalrook}.  We let 
$\lambda(\widehat{\rho}) = w_1 \dots w_n$, where 
\begin{equation}
w_i = \# \left\{ \begin{array}{c} \text{squares in Column $i$ weakly below} \\ \text{rooks and laser endpoints}
\end{array}
\right\}
\end{equation}
for $1 \leq i \leq n$.
In the case of Figure~\ref{verticalrook}, we have that
$\lambda(\widehat{\rho}) = 28818825$.

For the map $\lambda$ to be well defined, we need to check that
$\lambda(\widehat{\rho}) = w_1 \dots w_n$ is a rook word for every valid rook placement $\widehat{\rho}$ on
$\widehat{B}_n$.  Since every rook above the bar fires a laser that ends below the bar,
we have that $1 \leq w_i \leq n$ for all $i$.  The first paragraph of the proof of 
Lemma~\ref{same-information} implies that every integer in the interval $[1, w_1]$ appears 
as a letter in $w_1 \dots w_n$.  We conclude that $\lambda(\widehat{\rho}) \in \Rook_n$, as desired.

\begin{theorem}
\label{dominance-theorem}
Let $G \subseteq {[n] \choose 2}$ be a graph.
The composition $\omega \circ \beta \circ \lambda \circ \widehat{\rho}$ induces a bijection from the regions
of $\Ish(G)$ to the regions of $\Shi(G)$ which preserves ceiling partitions and dominance.
\end{theorem}
\begin{proof}
As mentioned above, we immediately reduce to the case $G = K_n$, $\Ish(G) = \Ish(n)$, and
$\Shi(G) = \Shi(n)$.  

We begin by showing that $\lambda$ is a bijection.
By Proposition~\ref{parking-shi-characterization}, the Cycle Lemma,
and Lemma~\ref{rook-reformulation}, this will show that the given composition is a bijection.    If 
$w_1 \dots w_n$ is a rook word, we form a valid rook placement $\widehat{\rho}$ on $\widehat{B}_n$
as follows.  We call the rooks in this rook placement Rooks $1, 2, \dots, n$.
\begin{itemize}
\item Place Rook $1$ in Column $1$ at height $w_1$ (and hence below the bar).
\item For $1 < j \leq n$, assume inductively that we have placed Rooks $1, 2, \dots , j-1$ in
Columns $1, 2, \dots, j-1$ and fired lasers down from those rooks placed above the bar.  
If position $j$ of $w_1 \dots w_n$ contains the first occurrence of the letter $w_j$, then place
Rook $j$ in Column $j$ below the bar at height $w_j$.  Otherwise, let $i$ be the maximal position $< j$ such that
$w_i = w_j$.  Place Rook $j$ in Column $j$ above the bar at height $n+i$.  If $k$ is the 
minimal position such that $w_k = w_j$, fire a laser down from Rook $j$ ending in the row containing
Rook $k$.
\end{itemize}
Since $w_1 \dots w_n$ is a rook word, the rook placement $\widehat{\rho}$ is valid.
It follows that $\lambda(\widehat{\rho}) = w_1 \dots w_n$, so 
$\lambda: \{$valid rook placements on $\widehat{B}_n \} \xrightarrow{\sim} \Rook_n$
is a bijection.

Next, we check that the given composition preserves ceiling partitions.  
Let $R$ be a region of $\Ish(n)$.
By the construction of the 
map $\lambda$, the ceiling partition of $R$ equals the position partition of the rook word
$(\lambda \circ \widehat{\rho})(R)$.  Since the Cycle Lemma map $\beta$ preserves position 
partitions and the map $\omega$ of Proposition~\ref{parking-shi-characterization}
maps position partitions to ceiling partitions, we conclude that the composition
$\omega \circ \beta \circ \lambda \circ \widehat{\rho}$ preserves ceiling partitions.

Finally, we check that the given composition preserves the dominance of regions.  Let $R$ be a 
dominant region of $\Ish(n)$.  This means that every rook in the valid placement 
$\widehat{\rho}(R)$ is either above the bar or on the diagonal
$\{(i, i) \,:\, 1 \leq i \leq n\}$.  In turn, this implies that the rook word 
$(\lambda \circ \widehat{\rho})(R)$ is componentwise $\leq (1, 2, \dots, n)$.  This forces
$(\lambda \circ \widehat{\rho})(R)$ to be a parking function and fixed by $\beta$, so that 
$(\beta \circ \lambda \circ \widehat{\rho})(R) = (\lambda \circ \widehat{\rho})(R)$.
Moreover, the word $(\lambda \circ \widehat{\rho})(R) = w_1 \dots w_n$ has the property that 
for all $i$ occurring in $(\lambda \circ \widehat{\rho})(R)$, the first occurrence of $i$ is at position $i$.
The fact that positive accumulations of dots increase from left to right and occur to the left of $1$
in Ish ceiling diagrams implies that the position partition of $(\lambda \circ \widehat{\rho})(R)$
is nonnesting.  We conclude that the 
Shi ceiling diagram $(\pi, \Pi)$ of
$(\omega \circ \beta \circ \lambda \circ \widehat{\rho})(R) = (\omega  \circ \lambda \circ \widehat{\rho})(R)$
has permutation $\pi_1 \pi_2 \dots \pi_n = 1 2 \dots n$, as desired.
\end{proof}

Continuing the example given by Figures~\ref{ishdiagram} and \ref{verticalrook}, 
if $R$ is the Ish region whose ceiling diagram is shown in Figure~\ref{ishdiagram},
we computed that 
$(\lambda \circ \widehat{\rho})(R) = 28818825$.  The unique parking function in the $\ZZ_9$-orbit
of $28818825$ is $41131147 \in \Park_8$.  We conclude that 
$(\beta \circ \lambda \circ \widehat{\rho})(R) = 41131147$.  Computing the image of this parking function
under $\omega$ shows that 
the composition in Theorem~\ref{dominance-theorem} sends the Ish region $R$ to the Shi region
whose Shi ceiling diagram is $(\pi, \Pi)$, where
$\pi = 23415786 \in \symm_8$ and 
$\Pi = \{ \{1, 2, 5, 8\}, \{3\}, \{4, 6\}, \{7\} \}$.

It can be checked that the bijection $\omega \circ \beta \circ \lambda \circ \widehat{\rho}$ restricts to a bijection
from Ish to Shi regions with $n$ degrees of freedom.  On the level of ceiling diagrams,
this bijection is just the inverse map $\pi \mapsto \pi^{-1}$ on $\symm_n$.

\section{A bijection between the relatively bounded regions of $\Shi(G)$ and $\Ish(G)$ which preserves 
ceiling partitions}
\label{Bounded section}

Let $G \subseteq {[n] \choose 2}$ be a graph.
A region $R$ of a hyperplane arrangement $\A$ is {\sf relatively bounded} if the number of degrees of 
freedom of $R$ is minimal among the regions of $\A$.  (If $\A$ is an essential arrangement, this is
equivalent to $R$ being bounded.)  By slightly modifying the reasoning of the last section, we can
solve a weaker version of Problem~\ref{freedom-problem} by giving a bijection
between the relatively bounded regions of $\Shi(G)$ and $\Ish(G)$ which preserves ceiling partitions.
The full bijective solution to Problem~\ref{freedom-problem}, 
presented in the next section, will restrict to this bijection
on relatively bounded regions.  Since this more general bijection is much more complicated,
we will present this special case first.

Without loss of generality, assume that $G$ is nonempty (if $G = \emptyset$, then 
$\Shi(G) = \Ish(G) = \Cox(n)$).  
This means that the relatively bounded regions
of $\Shi(G)$ or $\Ish(G)$ have  $1$ degree of freedom.  In terms of ceiling diagrams,
 this means that a Shi ceiling diagram $(\pi, \Pi)$ corresponds to a relatively bounded region if and 
 only if $\Pi$ is connected and an Ish ceiling diagram $(\pi, \epsilon)$ corresponds to a relatively
 bounded region if and only if $\pi_1 = 1$ and $\epsilon_n > 0$.  
 In terms of parking functions and rook
 words, we have the following lemma.
 
 \begin{lemma}
 \label{prime-reformulation}
 The map $\omega$ of Proposition~\ref{parking-shi-characterization} restricts to a bijection from the set
 $\Park_n'$ of prime parking functions to the set of relatively bounded regions of $\Shi(n)$.
 The map $(\lambda \circ \widehat{\rho})$ 
 restricts to a bijection from the set of relatively bounded regions of $\Ish(n)$ to the set
 $\Rook_n'$ of prime rook words.
 \end{lemma}
 
 \begin{proof}
 If $(\pi, \Pi)$ is the ceiling diagram of a relatively bounded Shi region, then the set partition
 $\Pi$ is connected.  This means that for every initial segment $[1, i]$ of $[n]$ with $i > 1$, the number 
 of blocks of $\Pi$ containing at least one of the elements in $[1, i]$ is strictly less than $i$.  This implies
 that $\omega$ sends prime parking functions to relatively bounded Shi regions.
 
A region $R$ of $\Ish(n)$ is relatively bounded if and only if the Ish ceiling diagram $(\pi, \epsilon)$
of $R$ satisfies $\pi_1 = 1$ and $\epsilon_n > 0$.   In terms of rook placements, this means that
the rook in Column 1 of $\widehat{\rho}(R)$ is in Row 1 and that there are no rooks in
Row $n$ of $\widehat{\rho}(R)$.  This means that the rook word $(\lambda \circ \widehat{\rho})(R)$ is prime.
The fact that every prime rook word corresponds to a relatively bounded Ish region is proven using
the same line of reasoning as in the proof of Theorem~\ref{dominance-theorem}.
 \end{proof}

Lemma~\ref{prime-reformulation} implies the following proposition.

\begin{proposition}
\label{bounded-proposition}
Let $G \subseteq {[n] \choose 2}$ be a graph.  The composition
$\omega \circ \beta' \circ \lambda \circ \widehat{\rho}$ induces a bijection from the relatively bounded
regions of $\Ish(G)$ to the relatively bounded regions of $\Shi(G)$ which preserves ceiling partitions.
\end{proposition}

\begin{proof} 
Without loss of generality, we assume that $G = K_n$.  By Lemma~\ref{prime-reformulation} and the Prime
Cycle Lemma, the given composition is a bijection which preserves ceiling partitions.
\end{proof}

The bijection in Proposition~\ref{bounded-proposition} does {\it not} preserve dominance because 
the prime cycle map $\beta'$ does not preserve dominance at the level of words.  
Indeed, this bijection could not preserve dominance because $\Shi(3)$ has $2$ dominant relatively
bounded regions and $\Ish(3)$ has $3$ dominant relatively bounded regions.

\section{A bijection between the regions of $\Shi(G)$ and $\Ish(G)$ which preserves ceiling partitions
and degrees of freedom}
\label{Freedom section}

\subsection{Parking functions as labeled Dyck paths}

\begin{figure}
\centering
\includegraphics[scale = 0.5]{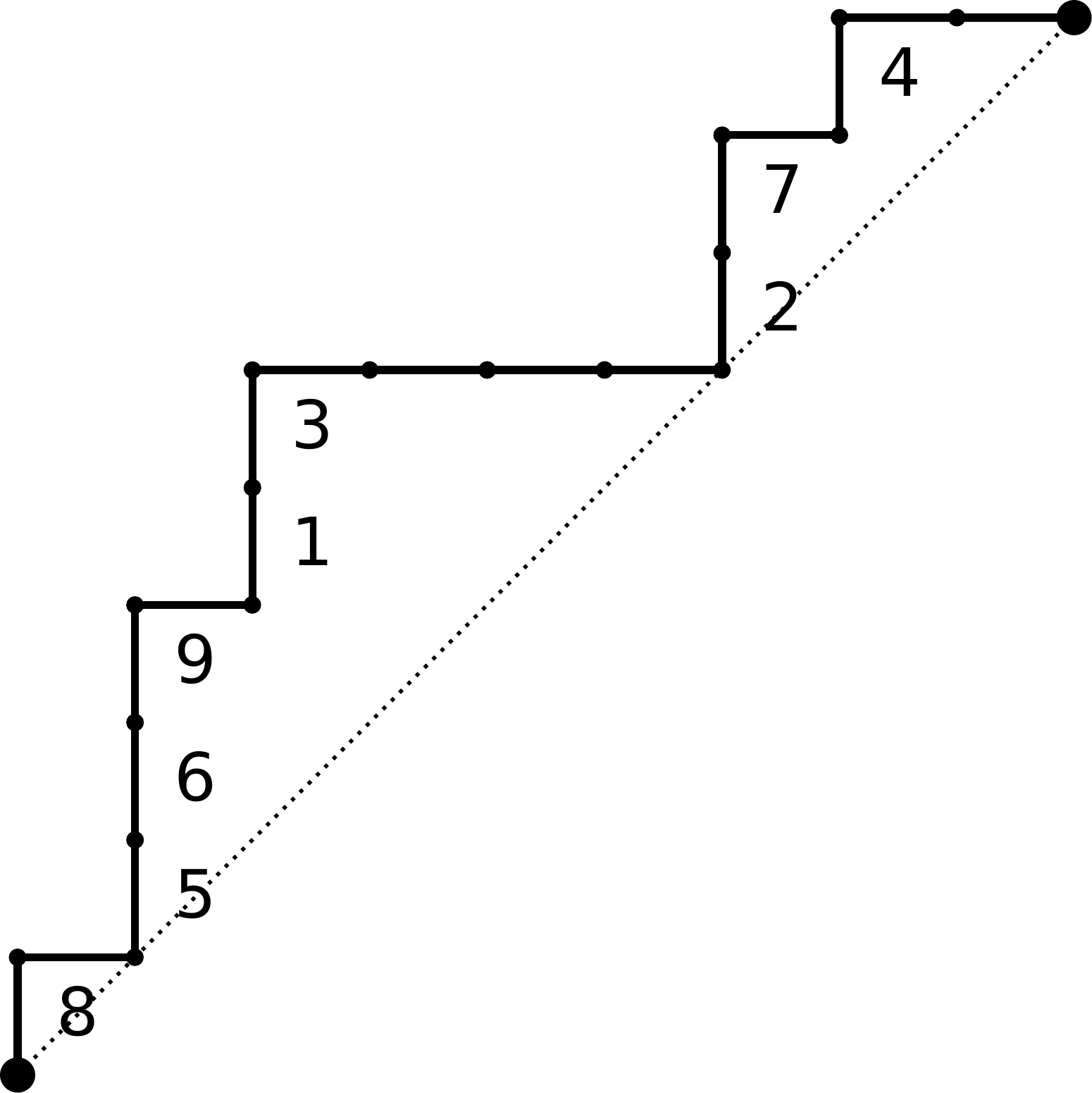}
\caption{A labeled Dyck path of size $9$.}
\label{dyckparking}
\end{figure}

In order to solve Problem~\ref{freedom-problem}, we will need to think of parking functions in a different
way.
A {\sf Dyck path of size $n$} is a lattice path in $\ZZ^2$ starting at $(0, 0)$, ending at $(n, n)$, consisting of
north steps $(0, 1)$ and east steps $(1, 0)$, and staying weakly above the diagonal $y = x$.  
Figure~\ref{dyckparking} shows a Dyck path of size $9$.  

A {\sf vertical run} in a Dyck path $D$ is a maximal contiguous sequence of north steps.  A {\sf labeled Dyck 
path} is a Dyck path $D$ of size $n$ with the north steps labeled by letters in $[n]$ such that every letter
appears exactly once as a label and the labels increase going up vertical runs.  A labeled Dyck path of size
$9$ is shown in Figure~\ref{dyckparking}.

There is a standard bijection between labeled Dyck paths of size $n$ and $\Park_n$ obtained by letting a 
labeled Dyck path $D$ correspond to the parking function $w_1 \dots w_n$, where $w_i$ is one plus
the $x$-coordinate of the column in which $i$ appears.  For example, the labeled Dyck path
in Figure~\ref{dyckparking} corresponds to the parking function
$373822712$.  Under this correspondence, the position partition of a parking function is the set partition
of $[n]$ defined by $i \sim j$ if and only if $i$ and $j$ label the same vertical run of $D$.

A {\sf return} of a (labeled) Dyck path $D$ is a diagonal point $(i, i)$ on $D$ other than
$(0, 0)$ or $(n, n)$.  The labeled Dyck path in Figure~\ref{dyckparking} has $2$ returns.  We call
$D$ {\sf prime} if it has no returns.  More generally, we call the prime Dyck paths defined by the returns of 
$D$ the {\sf prime components} of $D$.  The labeled Dyck path in Figure~\ref{dyckparking} has $3$ prime 
components.  The relationship between prime components and degrees of freedom is as follows.

\begin{lemma}
\label{prime-dyck}
Let $w = w_1 \dots w_n \in \Park_n$ and let $D$ be the corresponding labeled Dyck path.  If $\omega$ is the map
of Proposition~\ref{parking-shi-characterization}, the number of degrees of freedom of $\omega(w)$ equals
the number of prime components of $D$.
\end{lemma}

The proof of Lemma~\ref{prime-dyck} is left to the reader.  We will identify parking functions and labeled Dyck 
paths from now on.

\subsection{The $\gamma$ map}  Viewing parking functions as labeled Dyck paths, we will define a bijection
$\gamma: \{$regions of $\Ish(n)\} \rightarrow \Park_n$ which preserves ceiling partitions and degrees 
of freedom (and hence restricts appropriately to subgraphs $G$ of $K_n$).
It will be easier to describe the inverse map $\Park_n \rightarrow \{$regions of $\Ish(n)\}$; we  denote
 this map by $\delta$.
 
 Let $D$ be a parking function of size $n$, thought of as a labeled Dyck path.  To construct the 
 Ish diagram
 $\delta(D)$,
 we will perform a recursive procedure using the prime components of $D$.  The prime component containing
 the label $1$ will be treated in a special way.
 
 Let $C_1, C_2, \dots, C_d$ be the prime components of $D$ listed so that the prime component $C_1$ contains
 the label $1$ and the prime components are cyclically listed from left to right.  Let $k_i$ be the size 
 of the Dyck path underlying the prime component $C_i$.
 The map $\delta$ will treat the component $C_1$ differently from the other components.
 We initialize the 
image $\delta(D)$ to be the empty word.
 
 We start with the prime
 component $C_1$.  
 If $C_1$ only contains the label $1$ (i.e., if the Dyck path underlying $C_1$ has size $1$),
 we update $\delta(D)$ by setting $\delta(D) = 1$. 
  Otherwise, the prime component $C_1$ has labels other than $1$.
 Starting at the column of $C_1$ containing $1$, moving cyclically to the right within $C_1$,
 and ignoring the final (necessarily empty) column of $C_1$, we write either the lowest label in that column 
 or a diamond (`$\diamond$') if that column has no labels.  This gives
 a word of length $k_1-1$.  In the case of Figure~\ref{dyckparking}, this word is
 $1 \diamond \diamond$ $5$.  We replace $\delta(D)$ by this word with a single additional $\diamond$ at the end.
 At this point, the word $\delta(D)$ has length $k_1$.
 In the case of Figure~\ref{dyckparking} we have 
 $\delta(D) = 1 \diamond \diamond$ $5$ $\diamond$.  
 
 Next we look at the prime component $C_2$.  We build a 
 word of length $k_2$ as follows.  Reading the columns of $C_2$ from right to left, we record 
 either the lowest label in that column or a $\diamond$ if that column has no labels.  
 In Figure~\ref{dyckparking}, the word so obtained is given by $2$ $4$ $\diamond$.  Unlike in the case of $C_1$,
 there is no `cycling' involved in writing down this word; we start at the leftmost column of $C_2$ and move right.
 We replace $\delta(D)$ by the word obtained by inserting the size $k_2$ word for $C_2$ just after 
 the first letter of $\delta(D)$.  At this point $\delta(D)$ has length $k_1 + k_2$.  In the case of 
 Figure~\ref{dyckparking}, we have that $\delta(D) = 1$ $2$ $4$ $\diamond \diamond \diamond$ $5$ $\diamond$.
 
 At stage $i$ for $i > 1$, we start with a word $\delta(D)$ of length $k_1 + \cdots + k_{i-1}$.  We build a
 word of length $k_i$ from the component $C_i$ by reading the lowest labels in the columns of $C_i$
 from left to right, or a $\diamond$ if that column has no labels.  We replace $\delta(D)$ by the word
 of length $k_1 + \cdots + k_{i-1} + k_i$ formed by placing this word of length $k_i$ just after the 
 $(i-1)^{st}$ position of $\delta(D)$.  In the case of Figure~\ref{dyckparking}, this gives the word
$\delta(D) = 1$ $2$ $8$ $4$ $\diamond \diamond \diamond$ $5$ $\diamond$.

At this point, the first $d$ letters of $\delta(D)$ are not diamonds and give labels in the $d$ prime 
components of $D$, starting with $1$ read cyclically from left to right.  We replace $\delta(D)$ 
with the word obtained by cyclically rotating these first $d$ letters to the left by one position.  
In the example of Figure~\ref{dyckparking}, we have that $d = 3$ and
$\delta(D) = 2$ $8$ $1$ $4$ $\diamond \diamond \diamond$ $5$ $\diamond$.

At this stage, we know that $1$ occurs in position $d$ of $\delta(D)$.  We cyclically rotate the {\em entire} 
word $\delta(D)$ to the left until the only entries occurring before $1$ in $\delta(D)$ label prime 
components to the left of $1$ in $D$.  In Figure~\ref{dyckparking}, this gives 
$\delta(D) = 8$ $1$ $4$ $\diamond \diamond \diamond$ $5$ $\diamond$ $2$.

Finally, we replace $\delta(D)$ with the unique Ish ceiling diagram $(\pi, \epsilon)$ such 
that $\pi \in \symm_n$ agrees with $\delta(D)$ in the non-diamond positions and 
the ceiling partition of $(\pi, \epsilon)$ agrees with the ceiling partition of $D$.
In the case of Figure~\ref{dyckparking}, this yields 
$\delta(D) = (\pi, \epsilon)$, where
$\pi = 8 1 4 3 7 6 5 9 2$ and $\epsilon = 0 0 0 1 2 5 0 6 0$. 

The following result solves Problem~\ref{freedom-problem}.

\begin{theorem}
\label{freedom-theorem}
The map $\delta$ is invertible and its inverse $\gamma: \{$regions of $\Ish(n) \} \rightarrow \Park_n$
has the property that $\omega \circ \gamma$ is a bijection from regions of $\Ish(n)$ to regions of 
$\Shi(n)$ which preserves ceiling partitions and degrees of freedom.  Therefore,
if $G \subseteq {[n] \choose 2}$ is a graph, the composition $\omega \circ \gamma$ restricts to a bijection
from regions of $\Ish(G)$ to regions of $\Shi(G)$ which preserves ceiling partitions and degrees of freedom.
\end{theorem}

\begin{proof}
Let $D \in \Park_n$ be a parking function thought of as a labeled Dyck path with $d$ prime components.  
In the construction of the 
Ish ceiling diagram
$\delta(D)$, the diamonds will ultimately be replaced by letters with dots over them.  After the length
of $\delta(D)$ reaches $n$ and the first $d$ letters are cycled, the distance between $1$ and the 
final diamond of $\delta(D)$ is $n-d+1$.  This distance is preserved by the overall cyclic shift of 
$\delta(D)$ that follows.  
By Lemma~\ref{prime-dyck},
it follows that the final Ish ceiling diagram $\delta(D)$ has $d$ degrees of freedom,
so that $\delta$ preserves the degrees of freedom statistic.  The map $\delta$ preserves ceiling partitions
by construction.  We are reduced to showing that $\delta$ is a bijection.

The inverse map $\gamma$ of $\delta$ is defined by running the algorithm defining $\delta$ in reverse.  
We will describe how to construct the image $\gamma(R)$, where $R$ is the image $\delta(D)$ of the 
labeled Dyck path of Figure~\ref{dyckparking} and leave it to the reader to check that $\gamma$
is a well defined inverse of $\delta$.

Let $R$ be the region of $\Ish(n)$ whose Ish ceiling diagram is $(\pi, \epsilon)$, where
$\pi = 8 1 4 3 7 6 5 9 2$ and $\epsilon = 0 0 0 1 2 5 0 6 0$.  To construct $\gamma(R)$, we first make
note of the ceiling partition of $R$ (in this case 
$\{ \{1, 3 \}, \{2, 7\}, \{4\}, \{5, 8, 9\}, \{8\} \}$) and the number of degrees of freedom of $R$
(in this case $3$).  

Next, we initialize a word $w$
   by replacing every entry $\pi_i$ of $\pi$ with $\epsilon_i > 0$ with a $\diamond$.
In this case, we have the word $w = 8$ $1$ $4$ $\diamond \diamond \diamond$ $5$ $\diamond$ $2$.

Since $1$ never has dots above it in an Ish ceiling diagram, the number $1$ necessarily appears 
among the symbols of $w$.
If the region $R$ has $d$ degrees of freedom, we cycle the word $w$ to the right until
the number $1$ is in position $d$.  
We keep track of the number of times we needed to cycle to do this; call this number 
in $\{0, 1, \dots, d-1\}$
the {\sf cycle index}.
In our case, we cycle right once to get a word with $1$ in position $3$, yielding
$w = 2$ $8$ $1$ $4$ $\diamond \diamond \diamond$ $5$ $\diamond$ and the cycle 
index is $1$.
Since the number of degrees of freedom of an Ish region is $n+1$ minus (the number of entries weakly
between $1$ and the last dotted entry in the corresponding Ish ceiling diagram), this cycling procedure
never cycles any $\diamond$ symbols around the end of $w$ and 
results in a word whose first $d$ symbols are numbers rather than diamonds.  In our case, these 
symbols are the prefix $2$ $8$ $1$.

The next step is to rotate the first $d$ entries of $w$ one unit to the right (while leaving the 
remaining entires in $w$ unaltered).  At this point, the word $w$ starts with $1$ and the first
$d$ positions in $w$ are numbers.  In our case, we have
$w = 1$ $2$ $8$ $4$ $\diamond \diamond \diamond$ $5$ $\diamond$.

Using the word $w$, we will construct the prime components 
$C_d, C_{d-1}, \dots, C_1$
of the labeled Dyck path $\gamma(R)$ (with indices in that order), together
with the {\it cyclic} order in which they appear from left to right.  The {\it linear} order of these prime components
(and hence $\gamma(R)$ itself)
will be determined by the cycle index after the prime components are constructed.  The construction
of the final component $C_1$ will be different from the construction of the other prime components
$C_d, C_{d-1}, \dots, C_2$.

By construction, the numbers appearing in $w$ are precisely the minimal elements in the blocks
of the ceiling partition of $R$.  Above every number in $w$, draw a vertical run with steps labeled by 
elements of the corresponding ceiling partition block, together with a single east step.  Above every
diamond, draw a single east step.  In our case, we obtain the following figure.

\begin{center}
\includegraphics[scale = 0.5]{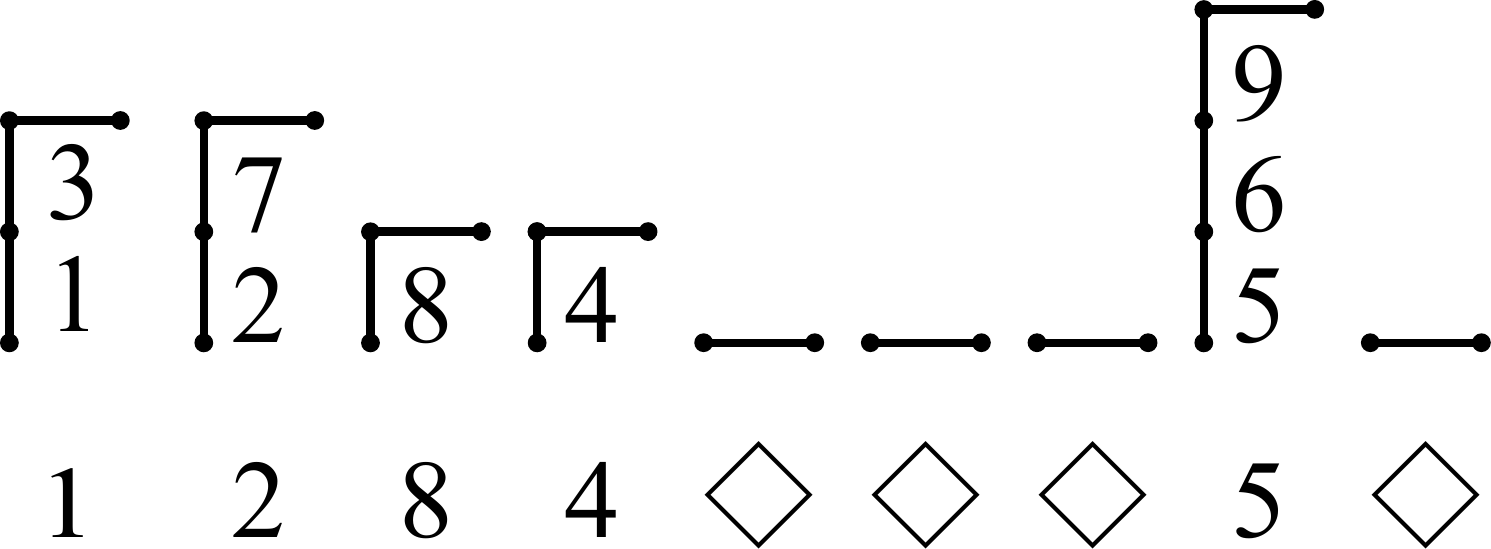}
\end{center}

To form the first prime component $C_d$, start at the $d^{th}$ position in this figure (which must necessarily 
contain a nonempty vertical run over a number).  Read from right to left, building up a labeled lattice
path starting on $y = x$ along the way, 
until one returns to the line $y = x$.  The labeled lattice path so obtained is the 
prime component $C_d$.  Delete the symbols involved with $C_d$ from the word $w$.  In our case, 
(where $d = 3$)
the prime component $C_3$ and the new word $w$ so obtained are shown below.  

\begin{center}
\includegraphics[scale = 0.5]{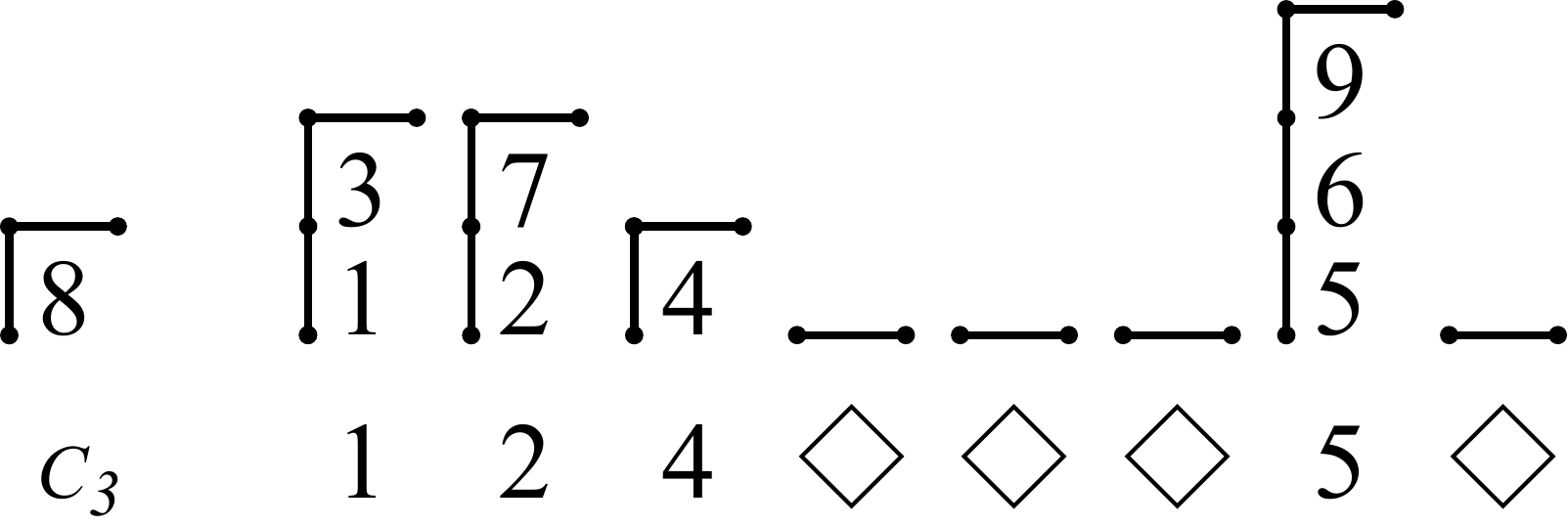}
\end{center}

To form the second prime component $C_{d-1}$, start at the $(d-1)^{st}$ position in this figure
(which must necessarily contain a nonempty vertical run over a number).  As with the component 
$C_d$, to form $C_{d-1}$ read the lattice path above $w$ from left to right until one returns to the diagonal.  
Delete the symbols involved in $C_{d-1}$ from $w$.  In our example, the prime components $C_3$ and
$C_2$ and the new word $w$ so obtained are as follows.

\begin{center}
\includegraphics[scale = 0.5]{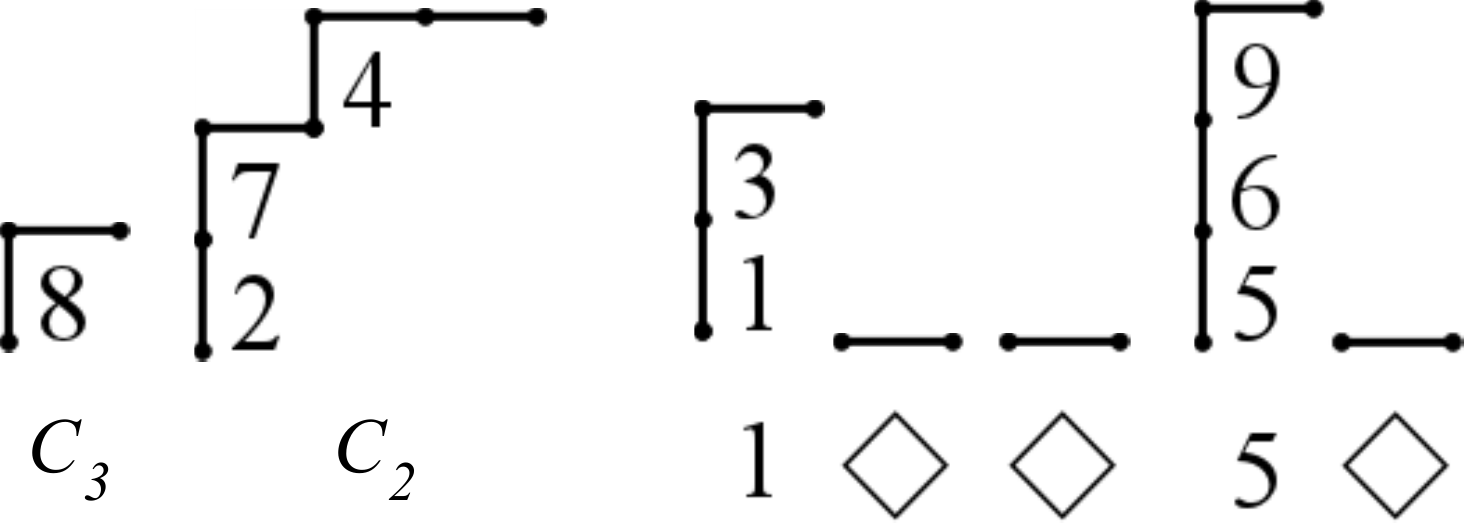}
\end{center}

For $2 \leq i \leq d$, to construct $C_i$ (having constructed $C_{i+1}, \dots, C_d$ and updated $w$ 
appropriately), start at the $i^{th}$ position in $w$ (which contains a number).  Read the lattice
path from left to right until it returns to the diagonal.  This gives the prime component $C_i$.
Delete the symbols involved in $C_i$ from $w$.

The construction of $C_1$ is different.  
If the word $w$ has length $1$ at this stage, then $w$ consists of the single letter $1$.  In this case,
we let $C_1$ be the Dyck path of size $1$ labeled by $1$.
If the word $w$ has length $> 1$, then  $w$ starts with $1$ and ends in
$\diamond$.  Delete the final $\diamond$ from $w$.  There exists a unique cyclic rotation
of the lattice path components above this shortened word which is of the form of 
a prime Dyck path missing its final east step.
This cyclic rotation (together with its labels and final east step) is the component $C_1$.

At this point, we have obtained the prime components $C_d, C_{d-1}, \dots, C_1$.  Suppose that the cycle index
is $k \in \{0, 1, \dots, d-1\}$.  The labeled Dyck path $\gamma(R)$ is defined to be the left-to-right
concatenation  $C_{d-k+1} C_{d-k+2} \dots C_d C_1 C_2 \dots C_{d-k}$.  In our example, 
the cycle index is $1$, so we have that $\gamma(R)$ is the left-to-right concatenation $C_3 C_1 C_2$.
This is precisely
the labeled Dyck path in Figure~\ref{dyckparking}.
\end{proof}

\section{Closing remarks}
\label{Closing remarks}

In this paper, we constructed bijections between the regions of the Shi and Ish arrangements which
preserved certain properties and statistics.  While the bijection preserving dominance
in Theorem~\ref{dominance-theorem}
was reasonably transparent and involved an application of the Cycle Lemma, the bijection
preserving degrees of freedom in Theorem~\ref{freedom-theorem} was more involved and ad hoc.
A possible method for obtaining a more conceptual proof of Theorem~\ref{freedom-theorem} (and
a better understanding of Shi/Ish duality) would be to find a rook word analog of a factorization property of
parking functions.

More precisely, if $u = u_1 \dots u_n \in \Park_n$, $v = v_1 \dots v_m \in \Park_m$, and $\{I / J\}$ is a 
two-block
ordered
set partition of $[n+m]$ with $|I| = n$ and $|J| = m$, we obtain a new parking function
$u(I) \otimes v(J) = w_1 \dots w_{n+m} \in \Park_{n+m}$ by the rule that the restriction of 
$w_1 \dots w_{n+m}$ to the positions in $I$ is $u_1 \dots u_n$ and the restriction of 
$w_1 \dots w_{n+m}$ to the positions in $J$ is $(v_1 + n) \dots (v_m + n)$.
(This is closely related to the `shuffle product' in a Hopf algebra structure on the set of parking functions
due to Novelli and Thibon \cite{NT}.)  Using the map $\omega$ of 
Proposition~\ref{parking-shi-characterization} to define the `degrees of freedom' of a parking function,
we see that the degrees of freedom of $u(I) \otimes v(J)$ is the sum of the degrees of freedom 
of $u$ and the degrees of freedom of $v$.

On the other hand, given any parking function $w = w_1 \dots w_n \in \Park_n$ with $d$
degrees of freedom, 
there exists a 
unique ordered set partition 
$\{ B_1 / B_2 /  \dots / B_d \}$ of $[n]$ and
unique prime parking functions $u_i \in \Park'_{|B_i|}$ such that
$w = u_1(B_1) \otimes \dots \otimes u_d(B_d)$.  For example, if $w \in \Park_9$ is the parking 
function whose labeled Dyck path is shown in Figure~\ref{dyckparking}, then
$d = 3$, 
$\{B_1 / B_2 / B_3 \} = \{ 8 / 1, 3, 5, 6, 9 / 2, 4, 7 \}$, 
$u_1 = 1$, $u_2 = 22111$, and $u_3 = 121$.  

The parking function factorization
$w = u_1(B_1) \otimes \dots \otimes u_d(B_d)$ also respects position partitions.  Namely, the position partition
of $w$ is the partition $\Pi$ of $[n]$ which refines the (unordered) set partition
$\{B_1, \dots, B_d\}$ such that the restriction of $\Pi$ to $B_i$ is order isomorphic to the
position partition of $u_i$.

Solving the following problem would give a new proof of Theorem~\ref{freedom-theorem} and an
Ish analog of the parking function factorization.

\begin{problem}
\label{factorization}
Let $n \geq 0$.  Give a bijection
\begin{equation}
\Rook_n \cong \biguplus_{\{B_1 / \dots / B_d\} \text{an ordered} \atop
\text{set partition of $[n]$}} \Rook'_{|B_1|} \times \dots \times \Rook'_{|B_d|}
\end{equation}
such that if $w \longleftrightarrow (u_1, \dots , u_d)$ and $(u_1, \dots, u_d)$ corresponds 
to $\{B_1 / \dots / B_d\}$, then
\begin{itemize}
\item  $w$ has $d$ degrees of freedom and 
\item the position partition 
of $w$ is the unique partition $\Pi$ of $[n]$ such that $\Pi$ refines $\{B_1, \dots, B_d \}$ and the restriction
of $\Pi$ to $B_i$ order isomorphic to the position partition of $u_i$.  
\end{itemize}
Here we define the 
``degrees of freedom" of a rook word to be the degrees of freedom of the corresponding Ish region
(via the map $\lambda \circ \widehat{\rho}$).
\end{problem}

In solving Problem~\ref{factorization}, it may be useful to have an intrinsic notion of the degrees of
 freedom of a rook word.
 Define the {\sf tail} of a rook word $w = w_1 \dots w_n \in \Rook_n$ to be the maximal 
 set of integers of the form $[j, n]$ for $w_1 +1 \leq j \leq n+1$
 such that every integer in $[j, n]$ appears as a letter of $w$.
We leave it to the reader to check that the number of degrees of freedom of $w$ is $w_1$ plus the 
number of integers in the tail of $w$.  For example, $211634 \in \Rook_6$ has tail $\{6\}$ and 
$2+1 = 3$ degrees of freedom.

\section{Acknowledgements}

The authors thank Drew Armstrong and Igor Pak for helpful conversations.  
E. Leven was partially supported by NSF Grant DGE-1144086.
B. Rhoades was partially
supported by NSF Grant DMS-1068861.  A. T. Wilson was partially supported by
the Department of Defense, Air Force Office of Scientific Research, National 
Defense Science and Engineering Graduate (NDSEG)
Fellowship, 32 CFR 168a.

\end{document}